\documentclass[a4paper,USenglish,cleveref, autoref, thm-restate]{lipics-v2021}
\hideLIPIcs  

\makeatletter
\renewenvironment{abstract}{%
  \vskip\topmattervskip\bigskipamount
  \noindent
  \rlap{\color{lipicsLineGray}\vrule\@width\textwidth\@height1\p@}%
  \hspace*{7mm}\fboxsep1.5mm\colorbox[rgb]{1,1,1}{\raisebox{-0.4ex}{%
    \large\selectfont\sffamily\bfseries\abstractname}}%
  \vskip3\p@
  \fontsize{9}{12}\selectfont
  \noindent\ignorespaces}
  {% (ccsdesc, keywords, and acknowledgements blocks removed)
   \ifx\@hideLIPIcs\@undefined
     \ifx\@DOIPrefix\@empty\else
       \vskip\topmattervskip\baselineskip\noindent
       \doiHeading\href{https://doi.org/\@lipicsdoi}{\@lipicsdoi}%
     \fi
   \fi
   \ifx\@category\@empty\else
     \vskip\topmattervskip\baselineskip\noindent
     \categoryHeading\@category
   \fi
   \ifx\@relatedversion\@empty\else
     \vskip\topmattervskip\baselineskip\noindent
     \relatedversionHeading\ifx\authoranonymous\relax\textcolor{red}{Anonymous related version(s)}\else\@relatedversion\fi
   \fi
   \ifx\@supplement\@empty\else
     \vskip\topmattervskip\baselineskip\noindent
     \supplementHeading\ifx\authoranonymous\relax\textcolor{red}{Anonymous supplementary material}\else\@supplement\fi
   \fi
   \ifx\@funding\@empty\else
     \vskip\topmattervskip\baselineskip\noindent
     \fundingHeading\ifx\authoranonymous\relax\textcolor{red}{Anonymous funding}\else\@funding\fi
   \fi
   \protected@write\@auxout{}{\string\gdef\string\@pageNumberEndAbstract{\thepage}}%
  }
\makeatother

\usepackage[dvipsnames]{xcolor}
\usepackage[ruled, linesnumbered]{algorithm2e}
\usepackage{amsfonts}
\usepackage{epstopdf}
\usepackage{tikz}
\usetikzlibrary{calc}
\usepackage{ifthen}
\usepackage{tcolorbox}
\usepackage{framed}
\usepackage{xifthen}
\usepackage{tabularx}
\usepackage{todonotes}  % \usepackage[disable]{todonotes}
\usepackage{amsopn}
\usepackage{derivative}
\usepackage{mathtools}
\usepackage[T1]{fontenc}
\usepackage{subcaption}
\usepackage{xspace}
\usepackage{dirtytalk}

\definecolor{nicered}{RGB}{204,0,0}
\definecolor{lightblue}{RGB}{153,204,255}

\usetikzlibrary{shapes.misc}

\tikzstyle{vertex}=[thin,circle,inner sep=0.cm, minimum size=1.7mm, fill=black, draw=black]%normal
 \tikzstyle{svertex}=[thin,circle,inner sep=0.cm, minimum size=1.3mm, fill=black, draw=black]%small
 \tikzstyle{bvertex}=[thin,circle,inner sep=0.cm, minimum size=1.7mm, fill=lightblue, draw=lightblue]%blue
 \tikzstyle{rvertex}=[thin,circle,inner sep=0.cm, minimum size=1.7mm, fill=nicered,draw=nicered]%red
 \tikzstyle{evertex}=[thin,circle,inner sep=0.cm, minimum size=1.7mm, fill=none,draw=black]%empty

 \tikzstyle{edge}=[thick, draw = gray]
 \tikzstyle{tedge}=[ultra thick, draw = black]%thick
 \tikzstyle{tredge}=[ultra thick, draw = nicered]
 \tikzstyle{tbedge}=[ultra thick, draw=lightblue]
 \tikzstyle{redge}=[thick, draw = nicered]
 \tikzstyle{bedge}=[thick, draw = lightblue]

\newcommand{\case}[1]{\medskip \noindent \textcolor{lipicsGray}{{\sffamily\bfseries Case #1.}}}

\newcommand{\maxdeg}{\Delta}

\newcommand{\mindeg}{\delta}

\newcommand{\chrom}{\chi}
\newcommand{\chromG}[1]{\chrom(G)}

\newcommand{\cclass}[1]{\textnormal{\textsf{#1}}}
\newcommand{\partition}[0]{\textsc{Degenerate Partitioning}\xspace}
\newcommand{\ppartition}[0]{$p$-\textsc{Degenerate Partitioning}\xspace}

\newcommand{\partitionNotation}[1]{\mathcal{P}_{#1}}
\newcommand{\setsCanJoin}[1]{\textup{choices}_{\maxdeg}(p,#1)}
\newcommand{\kcolor}{$k$-\textsc{Colorability}\xspace}
\newcommand{\scolor}{$s$-\textsc{Colorability}\xspace}

\newcommand{\quasi}{K^*_{\Delta+1}}

\newcommand{\inducedcyclelengthG}[1]{c^*({#1})}
\newcommand{\bridgesG}[1]{b({#1})}
\newcommand{\minrubinnumberG}[1]{h({#1})}

\newcommand{\Erd}{Erd\H{o}s\xspace}
\newcommand{\Lovasz}{Lov{\'a}sz}

\newcommand{\bct}{block-cut-tree\xspace}

\DeclareMathOperator*{\low}{low}
\DeclareMathOperator*{\high}{high}
\newcommand{\vlow}{V_{\low}}
\newcommand{\vhigh}{V_{\high}}
\newcommand{\defvlowb}{$\vlow = \{v: 0 \leq \deg_G^\ell(v) <  \frac{\maxdeg+q}{2} \}$}
\newcommand{\defvlowG}[1]{$\vlow = \{v: 0 \leq \deg_{#1}^\ell(v) < \frac{\maxdeg+q}{2} \}$}
\newcommand{\defvhighb}{$\vhigh = \{v: \frac{\maxdeg+ q}{2} \leq \deg_G^\ell(v) \}$}
\newcommand{\defvhighG}[1]{$\vhigh = \{v: \frac{\maxdeg + q}{2}  \leq \deg_{#1}^\ell(v) \}$}

\newcommand{\boundvlowb}{$|\vlow| \geq \frac{q}{\maxdeg+q}|V(G)|$}
\newcommand{\boundvlowshort}{\frac{|V(G)|}{\maxdeg+1}}
\newcommand{\boundvlowshortb}{\frac{q}{\maxdeg+q}|V(G)|}
\newcommand{\boundvlowG}[1]{$|\vlow| \geq \frac{q}{\maxdeg+q}|V(#1)|$}
\newcommand{\boundvlowshortG}[1]{\frac{q}{\maxdeg+q}|V(#1)|}

\title{Greedy Coloring and Simple Exponential Strengthenings of Brooks-Type Theorems}

\author{Alex Crane}{University of Utah, USA}{alex.crane@utah.edu}{https://orcid.org/0009-0004-5466-3181}{}

\author{Carl Feghali\footnote{It is with great sadness that we note the passing of Carl Feghali earlier this year. We are grateful for the opportunity to have worked with him on this project.}}{CNRS, ENS de Lyon, Université Claude Bernard Lyon 1, LIP UMR 5668, Lyon, France}{}{0000-0001-6727-7213}{}

\author{Felicia Lucke}{ENS de Lyon, LIP UMR 5668, Lyon, France}{felicia.lucke@ens-lyon.fr}{https://orcid.org/0000-0002-9860-2928}{supported by SNSF Postdoc Mobility Grant 230578.}

\author{Blair D.\ Sullivan}{University of Utah, USA, Collegium de Lyon, ENS de Lyon, LIP UMR 5668, Lyon, France}{sullivan@cs.utah.edu}{https://orcid.org/0000-0001-7720-6208}{supported in part by a grant from the Fulbright U.S. Scholar Program, sponsored by the U.S. Department of State and the Franco-American Fulbright Commission.}

\author{Madelyn J.\ Warr}{University of Utah, USA}{}{https://orcid.org/0009-0006-8042-8497}{supported by funding from the Undergraduate Research Opportunities Program at the University of Utah.}

\authorrunning{A.\ Crane, C.\ Feghali, F.\ Lucke, B.\,D.\ Sullivan, M.\,J.\ Warr}

\Copyright{Jane Open Access and Joan R. Public} %TODO mandatory, please use full first names. LIPIcs license is "CC-BY";  http://creativecommons.org/licenses/by/3.0/

 %TODO mandatory; please add comma-separated list of keywords

\category{} %optional, e.g. invited paper

\relatedversion{} %optional, e.g. full version hosted on arXiv, HAL, or other respository/website
\acknowledgements{I want to thank \dots}%optional

\nolinenumbers %uncomment to disable line numbering

\EventEditors{John Q. Open and Joan R. Access}
\EventNoEds{2}
\EventLongTitle{42nd Conference on Very Important Topics (CVIT 2016)}
\EventShortTitle{CVIT 2016}
\EventAcronym{CVIT}
\EventYear{2016}
\EventDate{December 24--27, 2016}
\EventLocation{Little Whinging, United Kingdom}
\EventLogo{}
\SeriesVolume{42}
\ArticleNo{23}
\begin{document}

\maketitle

\begin{abstract}
    Brooks' Theorem is a foundational result in graph coloring: if a graph is connected, not a clique, and has maximum degree $\maxdeg \geq 3$, then it admits a proper $\Delta$-coloring.
    In this paper, we give a new proof of a much stronger result: every such graph admits exponentially many proper $\maxdeg$-colorings.
    Exponential lower bounds have been obtained previously by applying the Combinatorial Nullstellensatz. In contrast, our proof is via a surprisingly simple and constructive analysis of a textbook greedy coloring algorithm. 
    The entire argument is accessible to non-specialists.

    We extend our techniques to provide quantitative lower bounds on the number of colorings in several more general settings, namely $\maxdeg$-length list coloring, degree-length list coloring (degree-choosability), and degenerate partitioning.
    In the latter setting, one is asked to partition
    the vertices such that each partite set
    induces a subgraph of bounded degeneracy. In contrast to the other coloring variants considered, no clean Brooks-type characterization is possible, since it is known that for some combinations of parameters determining the existence of even one valid partition is NP-hard. We give a new reduction which significantly extends the parameter regime for which determining the existence of a valid partition is NP-hard.

    Finally, we give a new proof of the classic degree-choosability theorem, generalizing the techniques behind a recent proof of Brooks' Theorem.

\end{abstract}

\newpage
\section{Introduction}

We study Brooks' Theorem,
a foundational result in the field of graph coloring. Some basic definitions are required. A proper $k$-coloring of a graph $G$ is an assignment of at most~$k$ colors to the vertices of $G$ such that no pair of adjacent vertices receives the same color. The chromatic number $\chi(G)$ is the least $k$ such that $G$ admits a proper $k$-coloring.

To introduce Brooks' Theorem and our techniques, it is useful to first consider the simplest textbook result in graph coloring: for every graph $G$ with maximum degree $\maxdeg$, $\chi(G) \leq \maxdeg + 1$.
The easiest argument is algorithmic: color the vertices in any fixed order, assigning to each vertex any color not appearing on any of its previously colored neighbors. Because each vertex has at most $\maxdeg$ neighbors and we have $\maxdeg + 1$ colors to choose from, this procedure is guaranteed to produce a proper coloring.
The algorithm we have just described is known as \emph{greedy coloring}, and will play a prominent role in this paper.

A natural and very old question is to determine for which graphs the bound $\chromG{G} \leq \maxdeg + 1$ is tight. Two obvious examples are cliques and odd cycles.
In 1941, Brooks showed that these are in fact the only examples.

\begin{theorem}[Brooks' Theorem~\cite{brooks1941colouring}]
    Let $G$ be a connected graph which is neither a clique nor an odd cycle. Then $G$ admits a proper $\maxdeg$-coloring.
\end{theorem}

\noindent
Numerous proofs are known, and the theorem has been generalized in many beautiful directions. For an introduction, see the survey of Cranston and Rabern~\cite{cranston2015brooks}, as well as the more recent textbook by Stiebitz, Schweser, and Toft~\cite{stiebitz2024brooks}.

In this paper, we will investigate generalizations of Brooks' Theorem which ask \emph{how many} proper $\maxdeg$-colorings exist.
We will be particularly interested in exponential lower bounds on the number of proper $\maxdeg$-colorings. No such bound can hold in general, even if we exclude cliques and odd cycles, as it is easy to check that the remaining connected graphs of maximum degree $2$ (paths and even cycles) admit only two proper $\maxdeg$-colorings.
It turns out that these are the only counterexamples to an exponential lower bound.

\begin{theorem}[\cite{dahlberg2023polynomial}]
    Let $G$ be a connected graph with maximum degree $\maxdeg \geq 3$, that is not a clique. Then $G$ admits at least $\maxdeg^{\frac{|V(G)|(\maxdeg - 2)}{2(\maxdeg - 1)}}$ proper $\Delta$-colorings.
\end{theorem}

\noindent
The proof given in~\cite{dahlberg2023polynomial} is by a straightforward application of a counting variant of the Combinatorial Nullstellensatz~\cite{alon1999combinatorial,alon1993covering}, a famously powerful but non-constructive tool.
Our first contribution is a simple and constructive proof of a weaker, but still exponential, bound.

\begin{restatable}[Simplified presentation of~\Cref{thm:quantbrooks}]{theorem}{quantitativebrookstheorem}\label{thm:basic-quant-brooks}
    Let $G$ be a connected graph with maximum degree $\maxdeg \geq 3$, that is not a clique. Then $G$ admits at least $\max\{\lceil\frac{\maxdeg}{2}\rceil^{\Omega(|V(G)|/\maxdeg)}, 2^{\Omega(|V(G)|)}\}$ proper $\maxdeg$-colorings.
\end{restatable}

\noindent
To see the basic idea, consider the easier setting in which we are allowed $k = \maxdeg + 2$ colors.
For this value of~$k$, every vertex considered during greedy coloring is guaranteed to have at least two available choices, so
branching on these choices yields a decision tree with $2^{|V(G)|}$ leaves, and thus an exponential number of proper $(\Delta + 2)$-colorings.
When $k = \Delta$, we achieve~\Cref{thm:basic-quant-brooks} by combining a more careful analysis of greedy coloring with a structural result due to~\Lovasz~\cite{lovasz1975three}.
Though we use~\Lovasz's result in a black-box manner, its proof requires only a short analysis of the block-cut-tree of $G$. Thus, our proof of~\Cref{thm:basic-quant-brooks} is easily accessible to readers with an elementary training in graph theory.

Another major thrust in the literature on Brooks' Theorem considers generalized notions of coloring. 
The Combinatorial Nullstellensatz techniques of~\cite{dahlberg2023polynomial} extend to some of these generalizations, and we explore how far we can push our quantitative analysis of greedy coloring.
For example, Brooks' Theorem is known to hold even for list coloring, as long as each list has length at least $\maxdeg$~\cite{erdos1979choosability,vizing1976coloring}. In this setting we achieve the same bound as in~\Cref{thm:basic-quant-brooks}, though we must work a bit harder to do so.

A more restrictive notion of list coloring arises when we assume that each list~$L(v)$ has length at least~$\deg(v)$. The degree-choosability theorem~\cite{borodin1977criterion,erdos1979choosability} is the analogue of Brooks' Theorem in this setting, and we again give a quantitative lower bound on the number of colorings. Our bound is exponential under mild structural assumptions. Furthermore, we provide several graph families on which our bound is tight, implying that assumptions similar to ours are necessary.

Beyond list coloring, another way to generalize graph coloring is to study degenerate partitioning. Formally, given $G$, an integer $s$, and $s$ integers $p_1, \ldots, p_s$, we can ask for a partition of the vertices into $s$ sets $V_1, \ldots, V_s$ such that each induced subgraph $G[V_i]$ has degeneracy\footnote{A graph has degeneracy $d$ if there exists an ordering of its vertices such that each vertex has at most $d$ neighbors appearing to its left. See~\Cref{sec:prelims} for formal definitions.} at most $p_i$. If $s = k$ and each $p_i = 0$, then we are asking for a proper $k$-coloring. 
Borodin~\cite{borodin1976decomposition} and independently Bollob\'as and Manvel~\cite{bollobas1979optimal} have given a Brooks-type result in this setting: a suitable partition is guaranteed to exist if $G$ is not a clique, $\maxdeg \geq 3$, $s \geq 2$, and $\sum_i p_i \geq \maxdeg - s$. We strengthen this result, giving a quantitative lower bound on the number of possible partitions.

All of our quantitative results are obtained via analyses of greedy coloring algorithms, and thus are easily accessible and constructive. We expect that our techniques can be pushed to more generalizations of graph coloring, and suggest that this will be an interesting direction for future work.

Beyond providing quantitative stregthenings of Brooks-type results, we make two further contributions. The first again concerns degenerate partitioning.
In this setting, while the results of~\cite{bollobas1979optimal,borodin1976decomposition} give sufficient conditions for the existence of a partition, a complete characterization in the style of Brooks' Theorem is unlikely.
Abu-Khzam, Feghali, and Heggernes~\cite{abu2020partitioning} gave evidence to this effect by showing that for some combinations of parameters, determining the existence of even one partition is \cclass{NP}-complete.
They then posed it as an open direction to determine the tractability frontier.
We make the first progress in this direction by 
showing that determining the existence of even one partition is \cclass{NP}-complete when $\Delta \geq s - 1 + \sum_i p_i + \Theta(k)$, for every $k \geq 3$. Our reduction is from \kcolor.

Finally, we contribute a new proof of the degree-choosability theorem, generalizing the techniques behind a recent proof of Brooks' Theorem given by Feghali~\cite{feghali2024dirac}. The proof is based on an analysis of chordal graphs, and extends a long line of discrete mathematics literature in which the goal is to see how far various proofs of Brooks' Theorem can be pushed; see the survey~\cite{cranston2015brooks}.

\smallskip
\noindent\textbf{Organization of the paper.} We introduce necessary definitions and basic lemmas in~\Cref{sec:prelims} before proving~\Cref{thm:basic-quant-brooks} in~\Cref{sec:quant-brooks}. We extend the result to $\maxdeg$-length list coloring in~\Cref{sec:quant-brooks-lists}. In~\Cref{sec:quant-degree-choice}, we give a quantitative strengthening of the degree-choosability theorem, before presenting our new proof of the theorem in~\Cref{sec:degree-choice-new-proof}. Finally, we give an analogue of~\Cref{thm:basic-quant-brooks} for degenerate partitioning in~\Cref{sec:quant-degenerate-partitioning}, before presenting our \cclass{NP}-hardness reduction in~\Cref{sec:hardness-degenerate}.

\section{Preliminaries}\label{sec:prelims}

We consider simple, loopless, and undirected graphs. For a graph $G$, we denote by~$V(G)$ the vertex set of $G$ and by $E(G)$ the edge set of $G$. For a vertex $v \in V(G)$, we denote by $\deg_G(v)$ the \emph{degree} of $v$ in $G$. We denote by $\mindeg(G)$ the \emph{minimum} degree of a vertex in $G$ and by $\maxdeg(G)$ the \emph{maximum} degree of a vertex in~$G$. We will just write $\deg(v)$, $\mindeg$, and $\maxdeg$ when the graph is clear. A graph is \emph{$d$-regular} if every vertex has degree exactly $d$. For a natural number $k \in \mathbb{N}$, we write~$[k]$ for the set $\{1, 2, \ldots, k\}$. 

For $k \in \mathbb N$ and a graph $G$, a $k$-\emph{coloring} of $G$ is a function $\phi: V(G) \rightarrow [k]$ assigning at most $k$ colors (here represented by integers) to the vertices of $G$. The coloring $\phi$ is~\emph{proper} if for every edge $uv$, we have $\phi(u) \neq \phi(v)$. The graph $G$ is \emph{$k$-colorable} if it admits a proper $k$-coloring.

Writing $2^{\mathbb N}$ for the (infinite) set consisting of all subsets of $\mathbb N$,
a \emph{list assignment} $L$ for a graph $G$ is a function $L: V(G) \rightarrow 2^\mathbb{N}$ assigning sets (which we call lists) of colors to the vertices of $G$. For a vertex $v$, we write $L(v)$ for the associated list. For an integer $k$, a $k$-\emph{list assignment} $L$ is a list assignment with the additional property that for every vertex $v$, $|L(v)| \geq k$.
If for every vertex $v$, $|L(v)| \geq \deg(v)$, then we say that $L$ is a \emph{degree-list assignment}.
For a graph $G$ and a list assignment $L$, an $L$-\emph{coloring} of $G$ is a proper coloring $\phi$ such that for each vertex~$v$, $\phi(v) \in L(v)$. Notice that if $L(v) = [k]$ for every vertex $v$, then an $L$-coloring is just a proper $k$-coloring. A graph $G$ is \emph{$k$-choosable} if for every $k$-list assignment $L$, there exists an $L$-coloring of $G$. A graph is \emph{degree-choosable} if for every degree-list assignment $L$, there exists an $L$-coloring of $G$.

Our arguments will rely on several more standard definitions, following the textbook by Diestel~\cite{diestel2017graph}. 
We denote the complete graph on $r$ vertices by $K_r$.
For a vertex set $S \subseteq V(G)$, we write $G - S$ for the subgraph induced by $V(G) \setminus S$. If $S$ is a singleton $\{v\}$, then we just write $G - v$. Similarly, for an edge~$uv$, we write $G - uv$ for the subgraph obtained by deleting $uv$.
In a connected graph $G$, a \emph{cut-vertex} is a vertex whose removal disconnects the graph, i.e., a vertex $v$ such that $G - v$ is disconnected. Similarly, a \emph{bridge} is an edge $uv$ such that $G - uv$ is disconnected.
A graph is \emph{$k$-connected} if it has at least $k + 1$ vertices and remains connected after removing any set of at most $k - 1$ vertices. 

A \emph{block} of a graph is a maximal $2$-connected subgraph. 
It is known that every graph admits a unique decomposition into blocks~\cite{diestel2017graph}. This decomposition is called the \emph{block-cut-tree} of~$G$, and is formed by treating each block of $G$ as a vertex, and adding an edge between two blocks if any cut-vertex of $G$ is contained in both blocks.
A \emph{leaf block} is a block corresponding to a leaf in the block-cut-tree. Equivalently, a leaf block is a block that contains exactly one cut-vertex of $G$.

Many of our arguments will rely upon (total) orderings of the vertices of a graph.
For an ordering $\sigma$ of $V(G)$, we say that $u$ is \emph{to the left} of~$v$ if~$u$ precedes~$v$ in~$\sigma$. If $u$ is to the left of~$v$ and $uv \in E(G)$, then $u$ is a \emph{left neighbor} of $v$.
We write $\deg_G^\ell(v)$ for the \emph{left degree} of $v$ in a fixed ordering $\sigma$, that is, for the number of left neighbors of $v$. For a fixed ordering, we say that $v$ is \emph{leftmost} if it is to the left of every other vertex, and similarly \emph{rightmost} if every other vertex is to the left of $v$.

An ordering $\sigma$ is a $d$-\emph{degeneracy ordering} if for every vertex $v$, we have $\deg_G^\ell(v) \leq d$. A graph $G$ is \emph{$d$-degenerate} if it admits a $d$-degeneracy ordering.
Observe that every graph of maximum degree $\Delta$ is $\Delta$-degenerate; any ordering suffices.
Similarly, if even one vertex~$v$ of a connected graph~$G$ has $\deg_G(v) < \maxdeg$, then $G$ is $(\maxdeg - 1)$-degenerate: an ordering can be formed by conducting a breadth-first traversal beginning from $v$, and then taking the reverse of this ordering so that $v$ is rightmost.

Another well-known fact is that every $d$-degenerate graph is $(d+1)$-colorable, and indeed $(d+1)$-choosable. To see this, we consider any $d$-degeneracy ordering $\sigma$ and \emph{color greedily} in this order. That is, we color the vertices in the order given by $\sigma$, assigning to each vertex any color not appearing on any of its left neighbors (any \emph{available} color). Because of the bound on left degrees, every vertex will have an available color. We call this procedure \emph{greedy coloring}, and it will be fundamental to many of our arguments. 

We conclude this section with a simple observation about degeneracy orderings that will be useful in several of our arguments.
\begin{observation}\label{obs-reg-vertremoval-deg}
    Let $G$ be a connected $d$-regular graph for some $d$. Let $v \in V(G)$. Then every connected component of $G-v$ is $(d-1)$-degenerate.
\end{observation}

\noindent
To see~\Cref{obs-reg-vertremoval-deg}, let $G'$ be any connected component of $G - v$, and consider any vertex $u \in V(G')$ which neighbors $v$ in $G$. Form an ordering of the vertices of $G'$ by breadth-first traversal from $u$, and then reverse this order so that $u$ is rightmost. Every vertex except $u$ has at least one neighbor to its right, and thus left-degree at most $d - 1$. Meanwhile, every neighbor of $u$ is to its left, but it has only $d - 1$ neighbors, so the result follows.

\section{A Quantitative Brooks' Theorem}\label{sec:quant-brooks}

We start with some lemmas that are important for many of our proofs.
In the following, we often work with a partition of the vertices, based on their left degree. The following lemma gives a lower bound on the number of vertices with small left degree.

\begin{lemma} \label{lem-highlow2}
   Let $G$ be a graph, $k \geq 2|E(G)|/|V(G)|$, and $q < k$. With respect to a fixed ordering of the vertices of $G$, define $\vlow = \{v : 0 \leq \deg_G^\ell(v) < \frac{k + q}{2}\}$, and $\vhigh = \{v : \frac{k + q}{2} \leq \deg_G^\ell(v)\}$.
    Then, $|\vlow| \geq \frac{q}{k+q}|V(G)|$.
\end{lemma}
\begin{proof}

Since $\vhigh$ and $\vlow$ partition $V(G)$, we have that $|\vhigh| + |\vlow| = |V(G)|$.
Now we suppose toward a contradiction that $|\vlow| < \frac{q}{k + q}|V(G)|$.
Then $$|\vhigh| = |V(G)| - |\vlow| > |V(G)| - \frac{q}{k + q}|V(G)|.$$
Since $\deg_G^\ell(v) \geq \frac{k + q}{2}$ for any $v \in \vhigh$, we get that
  \begin{align*}
    |E(G)| \geq \sum_{v \in \vhigh} \deg_G^\ell(v) &> \left(|V(G)| - \frac{q}{k + q}|V(G)|\right) \left(\frac{k+q}{2}\right) \\
    &= \left(\frac{k}{k+q}|V(G)|\right)\left(\frac{k + q}{2} \right) = \frac{k}{2}|V(G)|,
  \end{align*}
which contradicts that $|E(G)| \leq \frac{k}{2}|V(G)|$, by definition of $k$. Hence, $|\vlow| \geq \frac{q}{k+q}|V(G)|$.
\end{proof}

\begin{remark}
  Notice that the maximum degree $\maxdeg$ is always a valid choice for $k$ in~\Cref{lem-highlow2}, as it is always true that $|E(G)| \leq \frac{\maxdeg}{2}|V(G)|$. Thus, for any $q < \maxdeg$,~\Cref{lem-highlow2} implies that the number of vertices with left degree strictly less than $\frac{\maxdeg + q}{2}$ is at least $\frac{q}{\maxdeg + q}|V(G)|$.
\end{remark}

\noindent
We now show how~\Cref{lem-highlow2} can be used to obtain a lower bound on the number of $L$-colorings of $G$, for any $\maxdeg$-list assignment $L$, as long as $G$ is not $\maxdeg$-regular.

 \begin{lemma} \label{lem-kdeg2}
    Let $G$ be a $d$-degenerate graph of maximum degree~$\maxdeg\geq 3$, with $d,q < \maxdeg$. 
    Let~$L$ be any $\maxdeg$-list assignment for $G$.
    Then $G$ has at least 
    $\left(\left\lfloor\frac{\maxdeg-q}{2}\right\rfloor + 1\right)^{\boundvlowshortb}$ 
    $L$-colorings.
  \end{lemma}
  \begin{proof} Let $G$ be a $d$-degenerate graph for some $d < \maxdeg$.
  We fix a $d$-degeneracy ordering $\sigma = v_1, \dots, v_n$  of the vertices.
  We partition the vertices of $G$ based on their left degree.
  Let \defvhighb, and let 
  \defvlowb.
      By Lemma~\ref{lem-highlow2}, setting $k = \maxdeg$, we have \boundvlowb.

  We count the number of colorings by processing the vertices in the order of $\sigma$, going from left to right.
  For a vertex $v_i$, we can choose among at least $\maxdeg- \deg_G^\ell(v_i)$ colors.
  If $v_i \in \vhigh$, there are thus at least $\maxdeg- d\geq 1$ colors available.
   If $v_i \in \vlow$, there are strictly more than $\maxdeg - \frac{\maxdeg+q}{2}$ colors available. 
  Thus, for a vertex $v_i \in \vlow$, there are at least $\lfloor \maxdeg - \frac{\maxdeg+q}{2}\rfloor + 1$ colors available.
  
   Hence, the number of $L$-colorings is at least
  \[
    \left(\left\lfloor\maxdeg - \frac{\maxdeg+q}{2}\right\rfloor + 1\right)^{|\vlow|} * (\maxdeg-d)^{|\vhigh|} \geq \left( \left\lfloor\frac{\maxdeg-q}{2}\right\rfloor + 1\right)^{\boundvlowshortb}.
  \]
  \end{proof}

\begin{remark}
  Notice that for $q = 1$, after substituting and accounting for the parity of~$\Delta$ we obtain a bound of $\left(\left\lceil \frac{\maxdeg}{2}\right\rceil\right)^{\frac{1}{\maxdeg+1}|V(G)|}$. Meanwhile, for $q = \maxdeg - 2$, the bound becomes $2^{\frac{\maxdeg - 2}{2\maxdeg - 2}|V(G)|}$.
\end{remark}

\begin{remark}\label{rem:unique-cols} Note that throughout this paper, two colorings that lead to the same partition and differ only up to permutation of colors are considered distinct. To avoid this, we can color each of the first $\maxdeg$ vertices with a different color and only start counting colorings afterwards. This gives us a slightly weaker bound of $\left(\left\lfloor\frac{\maxdeg-q}{2}\right\rfloor + 1\right)^{\boundvlowshortb - \maxdeg}$ colorings. A similar analysis can be done for our later results, and we omit the details.
\end{remark}

\subsection{$\Delta$-Regular Graphs}\label{sec:delta-regular}
  To obtain an exponential lower bound on the number of proper $\maxdeg$-colorings, it remains only to handle $\maxdeg$-regular graphs. We are inspired by~\Lovasz's proof of Brooks' Theorem~\cite{lovasz1975three}, and in particular we will make use of~\Cref{lemma:lovasz-2-connected-lemma}. 

\begin{lemma}[\cite{lovasz1975three}]\label{lemma:lovasz-2-connected-lemma}
    Let $G$ be a $2$-connected graph with minimum degree $\mindeg \geq 3$ which is not complete. Then $G$ contains a pair of non-adjacent vertices $u_1, u_2$ with a common neighbor such that $G - \{u_1, u_2\}$ is connected.
\end{lemma}

\noindent
A sketch of the proof is as follows; see~\cite{lovasz1975three}. Since $G$ is not complete, let $x$ be any vertex with at least one non-neighbor. If $G - x$ is $2$-connected, then $x$ and any vertex at distance~$2$ from~$x$ satisfy the lemma. Then we assume that $G - x$ has a cut-vertex. The block-cut-tree of $G - x$ has at least two leaf blocks $B_1$ and $B_2$, with associated cut vertices $z_1$ and $z_2$ (it is possible that $z_1$ = $z_2$). Since $G$ is $2$-connected, $x$ must have some neighbor $u_1$ in $B_1 - z_1$ and another neighbor $u_2$ in $B_2 - z_2$. By definition of a leaf block, neither $u_1$ nor $u_2$ is a cut vertex in $G - x$. It is now easy to check that $G - \{u_1, u_2\}$ remains connected.

In our proof, we will use the vertices given by~\Cref{lemma:lovasz-2-connected-lemma} to find a vertex ordering with certain nice properties. These properties are captured by the following.
\begin{lemma}\label{lemma:non-cutset-ordering}
  Let $G$ be a connected graph of maximum degree $\maxdeg$, and let $X \subseteq V(G)$ be a set of vertices such that $G - X$ is connected, and $|X| \leq \maxdeg$.
  Then for any $v \in V(G) \setminus X$, there exists an ordering of the vertices of $G$ with the following properties:
  \begin{enumerate}
    \item Every vertex in $X$ is to the left of any vertex in $V(G) \setminus X$.
    \item The rightmost vertex is $v$.
    \item Every vertex $u \neq v$ has $\deg_G^\ell(u) < \maxdeg$.
  \end{enumerate}
\end{lemma}
\begin{proof}
  We begin by ordering the vertices of $G - X$ by breadth-first traversal from $v$, and then taking the reverse of this order so that $v$ is rightmost.
  We prepend the vertices of $X$ to the left end of the ordering. Thus, Properties (1) and (2) are immediate from the construction.
  Property (3) holds for $u \in X$ because each such $u$ has at most $|X| - 1 < \maxdeg$ vertices to its left, and hence $\deg_G^\ell(u) < \maxdeg$.
  For $u \in V(G) \setminus (X \cup \{v\})$, Property (3) holds because our breadth-first construction ensures that $u$ has at least one neighbor to its right, and hence at most $\maxdeg - 1$ neighbors to its left. 
\end{proof}

\noindent
We are now ready to prove our quantitative strengthening of Brooks' Theorem.
\begin{theorem}\label{thm:quantbrooks}
  Let $G$ be a connected graph of maximum degree $\maxdeg \geq 3$ that is not a complete graph, and let $q < \maxdeg$. Then, $G$ has at least~$\frac{1}{\maxdeg}\cdot\left(\left\lfloor\frac{\maxdeg-q}{2}\right\rfloor + 1\right)^{\frac{q}{\maxdeg + q}|V(G)| - 1}$ proper $\maxdeg$-colorings.
\end{theorem}
\begin{proof}
  Let $G$ be a graph of maximum degree~$\maxdeg$.
  If $G$ is $d$-degenerate for some $d < \maxdeg$, we apply Lemma~\ref{lem-kdeg2} where $L(v) = [\Delta]$ for all $v \in V(G)$.
Hence, we are left with the case where $G$ is $\maxdeg$-regular.
We distinguish two cases.

  \case{1} $G$ has a cut-vertex $v$.\\
  Let $G'_1, \dots , G'_r$ be the connected components of $G-v$.
    Let $G_1 = G[V(G'_1) \cup \{v\}]$ and let $G_2 = G[\bigcup_{i = 2}^{r}V(G'_i)\cup \{v\}]$.
    By \cref{obs-reg-vertremoval-deg}, we obtain that $G_1$ and $G_2$ are both $d_i$-degenerate for some $d_i < \maxdeg$.
  We set $L(u) = [\maxdeg]$ for every $u \in V(G)$ and apply Lemma~\ref{lem-kdeg2} to each $G_i$. We thus obtain that each $G_i$ has at least $\left(\left\lfloor\frac{\maxdeg-q}{2}\right\rfloor + 1\right)^{\frac{q}{\maxdeg + q}|V(G_i)|}$ proper $\maxdeg$-colorings.
  Since $G_1$ and $G_2$ share exactly one vertex $v$, we can only combine colorings of $G_1$ and $G_2$ in which $v$ has the same color in both colorings. Since we can permute color classes, a $\frac{1}{\maxdeg}$-fraction of the colorings assign $v$ a specific color $c$.
  Thus, for every fixed coloring of $G_2$, a $\frac{1}{\maxdeg}$-fraction of the colorings of $G_1$ give the same color to $v$. 
  Hence, the total number of proper $\maxdeg$-colorings of $G$ is at least 
  \[\left(\frac{1}{\maxdeg}\right)\cdot\left(\left\lfloor\frac{\maxdeg-q}{2}\right\rfloor + 1\right)^{\frac{q}{\maxdeg + q}(|V(G_1)|+|V(G_2)|)} \geq \left(\frac{1}{\maxdeg}\right)\cdot\left(\left\lfloor\frac{\maxdeg-q}{2}\right\rfloor + 1\right)^{\frac{q}{\maxdeg + q}(|V(G)| + 1)}\]

\case{2} $G$ is $2$-connected.\\
Let $v$ be a vertex with two non-adjacent neighbors $u_1, u_2$ such that $G - \{u_1, u_2\}$ is connected. 
Such vertices exist by Lemma~\ref{lemma:lovasz-2-connected-lemma}, which holds since $G$ is $\maxdeg$-regular for $\maxdeg \geq 3$.
Then let $\sigma$ be the vertex ordering given by~\Cref{lemma:non-cutset-ordering} for $X = \{u_1, u_2\}$, and assume without loss of generality that $u_1$ appears to the left of $u_2$.
We apply~\Cref{lem-highlow2} to $G$ with this ordering and with the parameter $k = \maxdeg$.
Thus, for the set $\vlow = \{v : \deg_G^\ell(v) < \frac{\maxdeg + q}{2}\}$, we have the bound $|\vlow| \geq \frac{q}{\maxdeg + q}|V(G)|$.

We now count colorings by an argument similar to that of~\Cref{lem-kdeg2}.
We process the vertices in the order given by $\sigma$, with the restriction that $u_1$ and $u_2$ receive the same color. This is always possible, since they are non-adjacent by definition and since they are the first two vertices to be colored by Property (1) of~\Cref{lemma:non-cutset-ordering}.
This restriction ensures that~$v$ always has an available color, since two of its neighbors receive the same color.
Every other vertex also always has an available color, by Property (3) of~\Cref{lemma:non-cutset-ordering}.
Furthermore, every vertex in $\vlow$ (except $u_2$, which is forced to receive the same color as $u_1$) has at least $\left\lfloor\frac{\maxdeg-q}{2}\right\rfloor + 1$ available colors.
Then the total number of proper colorings is at least
\[\left(\left\lfloor\frac{\maxdeg-q}{2}\right\rfloor + 1\right)^{|\vlow \setminus \{u_2\}|}
   \geq \left(\left\lfloor\frac{\maxdeg-q}{2}\right\rfloor + 1\right)^{\boundvlowshortG{G} - 1}.\]

The result now follows from a pessimistic combination of the bounds from these cases.
\end{proof}

  \begin{remark}
    Notice that when $q = 1$, after substituting and accounting for the parity of $\maxdeg$, we obtain the bound
    \[
      \left(\frac{1}{\maxdeg}\right)\cdot\left(\left\lceil \frac{\maxdeg}{2} \right\rceil\right)^{\frac{|V(G)|}{\maxdeg + 1} - 1} \geq \left(\frac{1}{2}\right)\cdot\left(\left\lceil \frac{\maxdeg}{2} \right\rceil\right)^{\frac{|V(G)|}{\maxdeg + 1} - 2} \geq \left(\left\lceil \frac{\maxdeg}{2} \right\rceil\right)^{\frac{|V(G)|}{\maxdeg + 1}-3}.
    \]
    Furthermore, when $q = \maxdeg - 2$, we obtain the bound
    \[
      \frac{1}{\maxdeg}\cdot2^{\frac{\maxdeg - 2}{2\maxdeg - 2}|V(G)| - 1} = 2^{\frac{\maxdeg - 2}{2\maxdeg - 2}|V(G)| -1 - \log \maxdeg}.
    \]
  \end{remark}

\subsection{Extension to $\maxdeg$-Length List Coloring}\label{sec:quant-brooks-lists}

In this section we will generalize~\Cref{thm:quantbrooks} to the setting of list coloring, where every list has length at least $\maxdeg$. It turns out that we can achieve the same bound as in~\Cref{thm:quantbrooks}.

\begin{theorem}
Let $G$ be a connected graph of maximum degree~$\maxdeg$, that is not a complete graph, and let $L$, be any $\maxdeg$-list assignment for $G$. Let $q < \maxdeg$. Then, $G$ has at least $\frac{1}{\maxdeg}\cdot\left(\left\lfloor\frac{\maxdeg-q}{2}\right\rfloor + 1\right)^{\frac{q}{\maxdeg + q}|V(G)| - 1}$ $L$-colorings.
\end{theorem}
\begin{proof}
  If $G$ is $d$-degenerate for some $d < \maxdeg$, then we apply Lemma~\ref{lem-kdeg2}.
  Hence, we are left with the case where $G$ is $\maxdeg$-regular. 
  Similarly, if even one vertex~$v$ has $|L(v)| > \maxdeg$, then we apply~\Cref{lemma:non-cutset-ordering}, with $X$ being the emptyset, to obtain a vertex ordering with $v$ rightmost. The result then follows by an argument conceptually identical to that in~\Cref{lem-kdeg2}, noting that $v$ always has an available color because its list is larger than its neighborhood.
  Hence, we may assume that every list has length exactly $\maxdeg$.
  
  If all lists are identical, then since we consider lists of length $\maxdeg$ we can use the same arguments as in Theorem~\ref{thm:quantbrooks}.
  We may therefore assume that not all lists are identical.
  Thus, since $G$ is connected there exist adjacent vertices $u,v \in V(G)$ with $c_u \in L(u) \setminus L(v)$ and $c_v \in L(v)\setminus L(u)$. We distinguish two cases based on whether both $u$ and $v$ are a cut-vertex.

\case{1} There are $u,v$ such that at most one of $u,v$ is a cut-vertex.\\
  Suppose that at least one of $u,v$ is not a cut-vertex, say $u$, so $G - u$ is connected. 
  Then we apply~\Cref{lemma:non-cutset-ordering} to obtain a vertex ordering $\sigma$ with $u$ leftmost, $v$ rightmost, and every vertex $w \neq v$ having $\deg_G^\ell(w) < \maxdeg$.
  With respect to this ordering, let $\vlow = \{w: 0 \leq \deg_G^\ell(w) < \frac{\maxdeg + q}{2}\}$. By applying~\Cref{lem-highlow2} with $k = \maxdeg$, we have that $|\vlow| \geq \frac{q}{\maxdeg + q}|V(G)|$.

  To count the number of colorings, we proceed as in the proof of~\Cref{lem-kdeg2}, processing the vertices from left to right in $\sigma$. 
  We fix the color of $u$ to be $c_u$. 
  For each following vertex, we check which colors are available.
  For each vertex $w \neq v$, there is a color available, by Property (3) of~\Cref{lemma:non-cutset-ordering}. Furthermore,~$v$ has an available color because we fixed color $c_u \notin L(v)$ for $u$, which neighbors $v$.
  
  For every $w \in \vlow$, except $u$ for which the color $c_u$ is fixed, there are at least~$\left\lfloor \maxdeg - \frac{\maxdeg + q}{2}\right\rfloor + 1$ colors available.
  It follows that the number of $L$-colorings of $G$ is at least
  \[
    \left(\left\lfloor \maxdeg - \frac{\maxdeg + q}{2}\right\rfloor + 1\right)^{|\vlow \setminus \{u\}|} \geq \left(\left\lfloor\frac{\maxdeg -q}{2}\right\rfloor + 1\right)^{\frac{q}{\maxdeg+q}|V(G)|-1}.
  \]

\case{2} Both $u,v$ are cut-vertices in $G$.\\
Suppose there are no vertices $u',v' \in V(G)$ such that we can apply Case~1.
This implies that every non-cut-vertex has the same list of colors as all of its neighbors.
Since we assume that not all lists are identical, we may assume that $G$ is not $2$-connected.

We consider the \bct of $G$.
Let $B$ be a leaf block of $G$. Hence, by definition, $B$ contains only one cut-vertex $x$.
Then, all neighbors of $x$ in $B$ have the same list as $x$. It follows that all vertices in $B$ have the same list.

  Let $G' = G - (B - x)$, that is, we delete all vertices in $B$ except $x$. Notice that $G'$ is connected, and in $G'$ $x$ has degree at most $\maxdeg-1$.
  By \cref{lemma:non-cutset-ordering}, with $X$ being the emptyset, we obtain a $(\maxdeg-1)$-degeneracy ordering of $G'$ where $x$ is rightmost. 
  Using~\Cref{lem-kdeg2}, $G'$ has at least $\left(\left\lfloor\frac{\maxdeg-q}{2}\right\rfloor + 1\right)^{\frac{q}{\maxdeg + q}|V(G')|}$ $L$-colorings.

  Consider a fixed $L$-coloring of $G'$. This coloring may be combined with any $L$-coloring of~$B$ to form an $L$-coloring of $G$, as long as $x$ receives the same color in both colorings.
  Since all vertices in $B$ have the same list, we may permute color classes to see that a $\frac{1}{\maxdeg}$-fraction of $L$-colorings of $B$ may be combined with any given $L$-coloring of $G'$.
  Again using~\Cref{lem-kdeg2} to give a bound on the number of $L$-colorings of $B$, we now have that the total number of $L$-colorings of $G$ is at least
  \[
    \left(\frac{1}{\maxdeg}\right)\left(\left\lfloor\frac{\maxdeg-q}{2}\right\rfloor + 1\right)^{\frac{q}{\maxdeg + q}(|V(G')| + |V(B)|)} = \left(\frac{1}{\maxdeg}\right)\left(\left\lfloor\frac{\maxdeg-q}{2}\right\rfloor + 1\right)^{\frac{q}{\maxdeg + q}(|V(G)|+1)}.
  \]
  The result now follows by a pessimistic combination of the bounds in these cases.
\end{proof}

\section{Degree-Length Lists}\label{sec:degree-choice}

The list coloring variant of Brooks' Theorem discussed in~\Cref{sec:quant-brooks-lists} raises a natural question: how long must the lists be to guarantee a coloring?
Recall from~\Cref{sec:prelims} that a \emph{list assignment}~$L$ for a graph $G$ is a function $L: V(G) \rightarrow 2^\mathbb{N}$ assigning sets (which we call lists) of colors to the vertices of $G$.
An $L$-coloring $\phi$ of $G$ is a proper coloring of $G$ such that for each vertex $v$, we have $\phi(v) \in L(v)$.

In this section, we will say that a list assignment $L$ is a \emph{degree-list assignment} if for every vertex $v$, we have $|L(v)| \geq \deg(v)$.
A graph $G$ is \emph{degree-choosable} if it admits an $L$-coloring for every degree-list assignment $L$.
Borodin~\cite{borodin1977criterion} and (independently) \Erd, Rubin, and Taylor~\cite{erdos1979choosability} have classified the degree-choosable graphs.
A \emph{Gallai tree} is a graph in which every block is either a clique or an odd cycle.

\begin{theorem}[Degree-Choosability~\cite{borodin1977criterion,erdos1979choosability}]\label{thm:degree-choosability}
    A connected graph is degree-choosable if and only if it is not a Gallai tree.
\end{theorem}

\noindent
It is well known that~\Cref{thm:degree-choosability} implies Brooks' Theorem; see, e.g.,~\cite{cranston2015brooks}.
In~\Cref{sec:quant-degree-choice}, we give a quantitative strengthening of~\Cref{thm:degree-choosability}. Our proof is once again based on an analysis of greedy coloring, but we will need to take more care when defining the order in which we color the vertices.
In~\Cref{sec:degree-choice-new-proof}, we give a new proof of~\Cref{thm:degree-choosability}, inspired by a recent proof of Brooks' Theorem given by Feghali~\cite{feghali2024dirac}.

\subsection{A Quantitative Degree-Choosability Theorem}\label{sec:quant-degree-choice}

It is useful to begin by discussing how the degree-choosability theorem may be proven.
Showing that Gallai trees are not degree-choosable follows by induction on the number of blocks, see e.g.,~\cite{cranston2015brooks}, so the main challenge is to prove that a connected graph is degree-choosable if it is not a Gallai tree.
In addition to~\cite{borodin1977criterion} and~\cite{erdos1979choosability}, several more proofs have appeared e.g.,~\cite{hladky2010brooks,kostochka1996colour}. A common strategy is to make use of a degree-choosable induced subgraph. We say that a \emph{Rubin graph} is an even cycle with at most one chord, and that $G$ has a \emph{Rubin subgraph} if it contains a Rubin graph as an induced subgraph.
The following observation has been made several times.

\begin{proposition}[\cite{erdos1979choosability}]\label{prop:rubin-subgraphs}
    Let $G$ be a connected graph. If $G$ has a Rubin subgraph, then $G$ is degree-choosable.
\end{proposition}

\noindent
For the sake of an informal outline, assume that Rubin graphs are degree-choosable; this is known by~\cite{erdos1979choosability}, and we will provide a strengthening of this result in~\Cref{prop:rubin-subgraphs-two-colorings}.
Then~\Cref{prop:rubin-subgraphs} follows by greedy coloring. Fix a Rubin subgraph $H$, and order the vertices of $G - H$ by increasing distance from $H$ in $G$. Next, greedily color the vertices in the reverse of this order, assigning to each $v$ any color from $L(v)$ not appearing on any previously colored neighbor of $v$. By our ordering, every vertex of $G - H$ has at least one uncolored neighbor when it is colored, and thus there is always a color available. The $L$-coloring of $G - H$ can then be extended to an $L$-coloring of $G$ by appealing to the degree-choosability of $H$.

\Erd, Rubin, and Taylor completed the proof of~\Cref{thm:degree-choosability} by showing that Gallai trees are exactly those graphs with no Rubin subgraph~\cite{erdos1979choosability}. This lemma has been proven several more times, e.g.,~\cite{entringer1985short,hladky2010brooks,rabern2013coloring}.
In this section, we show that by taking more care when defining our vertex ordering, the greedy-coloring-based argument yielding~\Cref{prop:rubin-subgraphs} can be extended to 
obtain a quantitative strengthening of~\Cref{thm:degree-choosability}.

To state the result, we need several notations. We say that a Rubin subgraph $H$ of a graph $G$ is \emph{minimum} if $G$ contains no Rubin subgraph with fewer vertices than $H$. If $G$ is not a Gallai tree, we write $\minrubinnumberG{G}$ for the number of vertices in a minimum Rubin subgraph of~$G$.
We also write $\inducedcyclelengthG{G}$ for the length of a longest induced (chordless) cycle in $G$, and~$\bridgesG{G}$ for the number of bridges in $G$. %Recall from~\Cref{sec:prelims} that the endpoints $u, v$ of a bridge $uv$ are the elements of a block of $G$ with cardinality exactly~$2$, and that in this manner bridges exactly characterize the cardinality~$2$ blocks of $G$.
We can now state our result.

\begin{restatable}{theorem}{quantitativedegreechoosability}\label{thm:quantitative-degree-choosability-theorem}
    Let $G$ be a connected graph which is not a Gallai tree, and $L$ be any degree-list assignment. Then $G$ admits at least $2^{\frac{|V(G)| - \minrubinnumberG{G} - \bridgesG{G}}{\inducedcyclelengthG{G}-1} + 1}$ $L$-colorings.
\end{restatable}

\noindent
At the end of this section we give several graphs on which the bound given by~\Cref{thm:quantitative-degree-choosability-theorem} is precisely tight. First, we focus on the proof.
The first intermediate step will turn out to be responsible for the additive constant appearing in the exponent in~\Cref{thm:quantitative-degree-choosability-theorem}.

\begin{restatable}{proposition}{rubinsubgraphstwocolorings}\label{prop:rubin-subgraphs-two-colorings}
    Let $H$ be a Rubin graph, and let $L$ be any degree-list assignment. Then~$H$ admits at least~$2$ $L$-colorings.
\end{restatable}

\noindent
The proof is by a straightforward but lengthy case-analysis; see~\Cref{appendix:degree-choice}. 
We now state the main step of the argument, which is analogous to~\Cref{lem-highlow2,lemma:non-cutset-ordering}.

\begin{restatable}{lemma}{degreechoosabilityorderinglemma}\label{lemma:degree-choice:main-ordering}
    Let $G$ be a connected graph which is not a Gallai tree, and let $H$ be any minimum Rubin subgraph of $G$. Then there exists an ordering $\sigma$ of $V(G)$ with the following properties.
    \begin{enumerate}
        \item Every vertex of $H$ is ordered to the left of any vertex of $G - H$.
        \item For every vertex $v$ of $G - H$, we have $\deg_G^\ell(v) \geq 1$.
        \item Let $\vlow \coloneqq \{ v \in V(G) \setminus V(H) : \deg_G^\ell(v) \geq 2\}$. Then $|\vlow| \geq \frac{|V(G)| - \minrubinnumberG{G} - \bridgesG{G}}{\inducedcyclelengthG{G} - 1}$.
    \end{enumerate}
\end{restatable}

\smallskip
\noindent
Our proof is algorithmic, and pseudocode is presented as Algorithm~\ref{alg:degree-choice:main-ordering-alg}.
\begin{proof}[Proof of~\Cref{lemma:degree-choice:main-ordering}]
    Let $H$ be a minimum Rubin subgraph of $G$.
    We will build a vertex ordering $\sigma$ from left-to-right. We begin by ordering the vertices of $H$ arbitrarily. We will order the remaining vertices by appending to the right of this order, so that Property (1) of the lemma is satisfied. 
    
    We now say that vertices already added to our ordering $\sigma$ are \emph{ordered} and all other vertices are \emph{unordered}. While not all vertices are ordered, we select an edge $uv$ with one ordered endvertex $u$ and one unordered endvertex $v$. There is always an available choice, since $G$ is connected.
    If the edge $uv$ is a bridge, we say that $uv$ was a \emph{bad choice}; otherwise, $uv$ was a \emph{good choice}. We now describe how to add vertices to our ordering according to these two possibilities.
    
    First suppose that $uv$ was a bad choice; see Lines~\ref{line:bridge-beginning}-\ref{line:bridge-end} in~\Cref{alg:degree-choice:main-ordering-alg}. We add $v$ to our ordering (on the right). 
    Notice that we have maintained Property (2) of the lemma, since $u$ was already ordered.

    Now suppose instead that $uv$ was a good choice; see Lines~\ref{line:good-choice-beginning}-\ref{line:good-choice-end} in~\Cref{alg:degree-choice:main-ordering-alg}. 
    Let $u, w_1, \ldots, w_p, v$ be a shortest $uv$-path in $G - uv$. Such a path exists by definition of a good choice ($uv$ is not a bridge).
    For each $i \in [p]$, we add $w_i$ to our order, unless it is already ordered.
    Lastly, we add $v$ to our order.
    Notice again that we have maintained Property~(2) of the lemma, since $w_{i - 1}$ is added to the order before $w_i$ for every $i \geq 1$.
    Moreover, $\deg_G^\ell(v) \geq 2$ since both $u$ and $w_p$ are neighbors of $v$ appearing to its left in $\sigma$.
    
    This completes the ordering algorithm. It remains to show that Property (3) is satisfied.
    We observe first that upon each bad choice, exactly one unordered vertex becomes ordered, and we cannot make more than $\bridgesG{G}$ bad choices.
    Then the remaining $|V(G)| - \minrubinnumberG{G} - \bridgesG{G}$ vertices are ordered via good choices. We have already shown that each good choice produces at least one vertex in~$\vlow$.
    Then it suffices to show that upon each good choice, at most $\inducedcyclelengthG{G} - 1$ unordered vertices become ordered.

    To see this, observe that upon a good choice $uv$, by combining the edge $uv$ with the path $u, w_1, \ldots, w_p, v$ we obtain a cycle in $G$. 
    Moreover, this cycle is induced, since $u, w_1, \ldots, w_p, v$ was chosen to be a shortest $uv$-path in $G - uv$.
    Thus, $p \leq \inducedcyclelengthG{G} - 2$, and the result follows.
\end{proof}
\begin{algorithm}
    \caption{The algorithm used for the proof of~\Cref{lemma:degree-choice:main-ordering}.}\label{alg:degree-choice:main-ordering-alg}
    \SetKwInOut{input}{Input}\SetKwInOut{output}{Output} 
    \SetKw{Continue}{continue}

    \input{ A connected graph $G$ which is not a Gallai tree, and a minimum Rubin subgraph $H$.}
    \output{ An ordering $\sigma$ of the vertices of $G$.}

    $\sigma \leftarrow$ An arbitrary ordering of $V(H)$\;
    \While{\text{Not all vertices are ordered}}{
        $uv \leftarrow$ an edge with $u$ ordered and $v$ unordered\;
        \If{$uv$ is a bridge}{\label{line:bridge-beginning}
            $\sigma \leftarrow \sigma + v$\;
            \Continue\;\label{line:bridge-end}
        }
        $u, w_1, \ldots, w_p, v \leftarrow$ a shortest $uv$-path in $G - uv$\;\label{line:good-choice-beginning}
        \For{$i = 1, \ldots, p$}{
            \If{$w_i$ is unordered}{
                $\sigma \leftarrow \sigma + w_i$\;
            }
        }
        $\sigma \leftarrow \sigma + v$\;\label{line:good-choice-end}
    }
    \Return $\sigma$\;

\end{algorithm}

\noindent
The proof of the main result,~\Cref{thm:quantitative-degree-choosability-theorem}, follows by combining~\Cref{lemma:degree-choice:main-ordering} and~\Cref{prop:rubin-subgraphs-two-colorings}. The argument is similar to the proof of~\Cref{lem-kdeg2}.

\begin{proof}[Proof of~\Cref{thm:quantitative-degree-choosability-theorem}.]
    Let $G$ be a connected graph which is not a Gallai tree, and let $H$ be a minimum Rubin subgraph. Let $\sigma$ be an ordering of the vertices given by~\Cref{lemma:degree-choice:main-ordering}. Let $\vlow$ be the set given by Property (3) of~\Cref{lemma:degree-choice:main-ordering}, and let $L_{G - H}$ be the restriction of $L$ to $V(G - H)$.
    Suppose we color the vertices of $G - H$ in the reverse of $\sigma$ (i.e., right-to-left). %By Property (1) of~\Cref{lemma:degree-choice:main-ordering}, we will not encounter any vertices of $H$ during this process.
    By Property (2) of~\Cref{lemma:degree-choice:main-ordering}, we have that for every vertex $v$ of $G - H$, there is some available color, and moreover for every $v \in \vlow$ there are at least $2$ colors available. 
    Thus, using Property (3), the number of possible greedy $L_{G-H}$-colorings of $G - H$ is at least
    \[
        2^{|\vlow|} \geq 2^\frac{|V(G)| - |V(H)| - \bridgesG{G}}{\inducedcyclelengthG{G} - 1} = 2^\frac{|V(G)| - \minrubinnumberG{G} - \bridgesG{G}}{\inducedcyclelengthG{G} - 1}.
    \]
    Now we complete the proof by showing that every fixed $L_{G - H}$-coloring of $G - H$ can be extended to form two unique $L$-colorings of $G$.
    Fix an $L_{G - H}$-coloring of $G-H$, and call this coloring $\phi$. For each $v \in V(H)$, let $L^*(v) \coloneqq L(v) \setminus \{ \phi(u) : u \in N_G(v) \setminus V(H) \}$. Notice that since $|L(v)| \geq \deg_G(v)$, we have $|L^*(v)| \geq \deg_H(v)$ for every $v \in V(H)$. Thus, $L^*$ is a degree-list assignment for $H$.
    Furthermore, any $L^*$-coloring of $H$ can be combined with $\phi$ to form a unique $L$-coloring of $G$.
    By~\Cref{prop:rubin-subgraphs-two-colorings}, there are at least $2$ $L^*$-colorings of $H$. This completes the proof.
\end{proof}

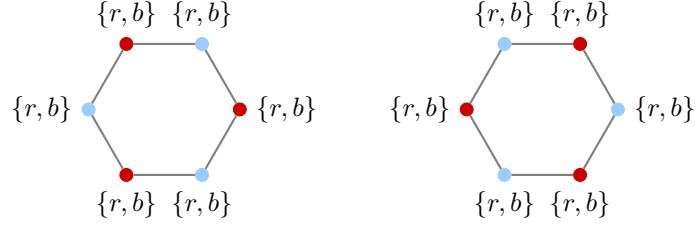
\begin{figure}
    \centering
    \begin{tikzpicture}
		\node[rvertex, label = {right:{$\{r,b\}$}}](u0) at (0:1){};
		\node[rvertex, label = {{$\{r,b\}$}}](u2) at (60*2:1){};
		\node[rvertex, label = {below:$\{r,b\}$}](u4) at (60*4:1){};
		\node[bvertex, label = {$\{r,b\}$}](u1) at (60*1:1){};
		\node[bvertex, label = {left:$\{r,b\}$}](u3) at (60*3:1){};
		\node[bvertex, label = {below:$\{r,b\}$}](u5) at (60*5:1){};
	\draw[edge](u0) -- (u1) -- (u2) -- (u3) -- (u4) -- (u5) -- (u0);

	\begin{scope}[shift = {(5,0)}]

		\node[rvertex, label = {{$\{r,b\}$}}](u0) at (60*1:1){};
		\node[rvertex, label = {left:{$\{r,b\}$}}](u2) at (60*3:1){};
		\node[rvertex, label = {below:$\{r,b\}$}](u4) at (60*5:1){};
		\node[bvertex, label = {$\{r,b\}$}](u1) at (60*2:1){};
		\node[bvertex, label = {below:$\{r,b\}$}](u3) at (60*4:1){};
		\node[bvertex, label = {right:$\{r,b\}$}](u5) at (60*0:1){};
	\draw[edge](u0) -- (u1) -- (u2) -- (u3) -- (u4) -- (u5) -- (u0);
	\end{scope}
\end{tikzpicture}
    \caption{\label{fig:even-cycle}The two ways of coloring an even cycle with identical lists of length two.}
\end{figure}%\vspace{25pt}

\begin{figure}
    \centering
    \begin{tikzpicture}[x=1.2cm]
		\node[rvertex, label = {{$\{r\}$}}](u0) at (0,0){};
		\node[bvertex, label = {{$\{r,b\}$}}](u1) at (1,0){};
		\node[rvertex, label = {{$\{r,b,w\}$}}](u2) at (2,0){};
		\node[evertex, label = {{$\{r,w\}$}}](u3) at (3,0){};
		\node[rvertex, label = {{$\{r,w\}$}}](u4) at (4,0){};
		\node[evertex, label = {{$\{r,w\}$}}](u5) at (5,0){};

	\draw[edge](u0) -- (u1) -- (u2) -- (u3) -- (u4) -- (u5) to [bend left = 30] (u2);

	\begin{scope}[shift = {(6,0)}]
		\node[rvertex, label = {{$\{r\}$}}](u0) at (0,0){};
		\node[bvertex, label = {{$\{r,b\}$}}](u1) at (1,0){};
		\node[evertex, label = {{$\{r,b,w\}$}}](u2) at (2,0){};
		\node[rvertex, label = {{$\{r,w\}$}}](u3) at (3,0){};
		\node[evertex, label = {{$\{r,w\}$}}](u4) at (4,0){};
		\node[rvertex, label = {{$\{r,w\}$}}](u5) at (5,0){};

	\draw[edge](u0) -- (u1) -- (u2) -- (u3) -- (u4) -- (u5) to [bend left = 30] (u2);

	\end{scope}
\end{tikzpicture}
    \caption{\label{fig:many-bridges}A graph and degree-list assignment such that only two colorings are possible.}
\end{figure}
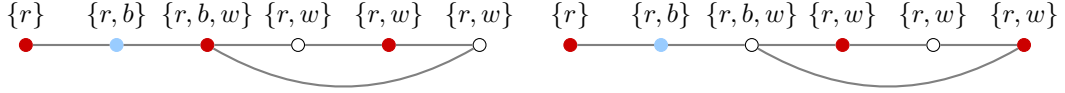

\noindent
To conclude, we remark upon the parameters $\minrubinnumberG{G}$, $\bridgesG{G}$, and $\inducedcyclelengthG{G}$ appearing in the bound of~\Cref{thm:quantitative-degree-choosability-theorem}. While there are likely graphs on which this bound is not tight, we know that it cannot be improved in general. We give two examples.

Consider first an even cycle $G$ with a list assignment $L$ consisting of identical lists of length~$2$. This graph admits exactly two $L$-colorings (see~\Cref{fig:even-cycle}), which matches the bound of~\Cref{thm:quantitative-degree-choosability-theorem}; notice that $\minrubinnumberG{G} = |V(G)|$.

A second example is formed by considering a path $v_1, \ldots, v_n$ on $n$ vertices, where $n\geq 6$ is an even integer. To form $G$, add an additional edge $v_{n - 3}v_n$. The case where $n = 6$ is depicted in~\Cref{fig:many-bridges}. This graph has $\minrubinnumberG{G} = \inducedcyclelengthG{G} = 4$, and $\bridgesG{G} = n - 4$. Hence,~\Cref{thm:quantitative-degree-choosability-theorem} guarantees only $2$ colorings.
A degree-list assignment $L$ matching this bound is defined as follows. Let $L(v_1) = \{r\}$, and for $1 < i < n - 3$, let $L(v_i) = \{r, b\}$. Let $L(v_{n - 3}) = \{r,b,w\}$, and finally for $i > n - 3$ let $L(v_i) = \{r, w\}$. In every $L$-coloring $\phi$ of~$G$, we have $\phi(v_i) = \{r\}$ for every odd $i < n - 3$, and $\phi(v_i) = b$ for every even $i < n - 3$. Since $n - 3$ is odd, it follows that $\phi(v_{n - 3}) \in \{b, w\}$. Thus, there are exactly two $L$-colorings, defined by the two ways of coloring the even cycle $v_{n-3},v_{n-2},v_{n-1},v_n,v_{n-3}$. 

\subsection{A New Proof of The Degree-Choosability Theorem}\label{sec:degree-choice-new-proof}

Recall from our discussion at the beginning of~\Cref{sec:quant-degree-choice} that the degree-choosability theorem (\Cref{thm:degree-choosability}) may be proven by first showing that Rubin graphs are degree-choosable (\cite{erdos1979choosability} or~\Cref{prop:rubin-subgraphs-two-colorings}), then by showing that it suffices for $G$ to contain a Rubin subgraph (\Cref{prop:rubin-subgraphs}~\cite{erdos1979choosability}), and finally by showing Gallai trees are exactly those connected graphs containing no Rubin subgraph.

Here, we will give a proof which omits the last step. In particular, we will instead prove~\Cref{thm:degree-choosability} by combining~\Cref{prop:rubin-subgraphs} with an examination of \emph{chordal} graphs, i.e., those graphs with no induced cycles of length greater than three. 
Our technique is inspired by a proof of Brooks' Theorem given recently by Feghali~\cite{feghali2024dirac}.

We begin by showing that every chordal graph has a Rubin subgraph, unless it is a Gallai tree. In fact, we show that these graphs must contain a \emph{diamond} (a $K_4$ minus one edge).

\begin{lemma}\label{lemma:chordal-graphs-degree-choosable}
    Let $G$ be a connected chordal graph. If $G$ is not a Gallai tree, then $G$ contains a diamond as an induced subgraph. 
\end{lemma}
\begin{proof}
    Since $G$ is not a Gallai tree, let $B$ be a block of $G$ which is neither a clique nor an odd cycle. Since $B$ is not a clique, let $u, v \in V(B)$ be a pair of vertices with distance $2$, and let $w$ be a mutual neighbor of $u$ and $v$. Let $P = u, \ldots, v$ be a shortest $uv$-path in $B - w$. Note that by our choice of $u$ and $v$, there is at least one internal vertex on $P$. Furthermore, since $P$ is a shortest path, it is an induced path.
    We observe further that $u, P, v, w, u$ is a cycle in $G$ of length at least $4$, and its only possible chords are edges linking $w$ with the internal vertices of $P$. 
    Since $G$ is chordal, every such chord exists, i.e., $w$ is adjacent to every vertex on $P$.
    Then let $u, z_1, z_2$ be the first three vertices of $P$ (it is possible that $z_2 = v$). Since $w$ is adjacent to each of these vertices, the induced subgraph $G[\{u, z_1, z_2, w\}]$ is a diamond.
\end{proof}

\noindent
Before completing our proof, we make two useful observations. Both identify additional hypotheses under which choosability is much easier to prove.

\begin{lemma}\label{lemma:degree-plus-one-choosable}
    Let $G$ be a connected graph, and let $L$ be a list assignment with $|L(v)| \geq \deg(v)$ for every vertex $v$. If there exists a vertex $w$ with $|L(w)| > \deg(w)$, then $G$ admits an $L$-coloring.
\end{lemma}
\begin{proof}
    Order the vertices left to right, beginning with $w$, in increasing order of distance from $w$. Then greedily color from right to left. This is possible, since every vertex except~$w$ has at least one neighbor ordered to its left, and thus an available color, while $w$ can be colored since its list is larger than its neighborhood.
\end{proof}
\begin{lemma}\label{lemma:2-connected-unequal-degree-choosable}
    Let $G$ be a 2-connected graph, and let $L$ be a list assignment with $|L(v)| \geq \deg(v)$ for every vertex $v$. If the lists assigned by $L$ are not all pairwise equal, then $G$ admits an $L$-coloring.
\end{lemma}
\begin{proof}
    Since $G$ is connected and the lists assigned by $L$ are not all pairwise equal, there exist adjacent vertices $u, v$ and a color $c$ with $c \in L(u) \setminus L(v)$. Order the vertices left to right, beginning with $v$, in increasing order of distance from $v$ in $G - u$. Order $u$ rightmost. Because $G$ is $2$-connected, all vertices except $v$ have at least one neighbor to the left. Fix color $c$ for $u$, and then color greedily from right to left. As its neighbor $u$ received color $c \notin L(v)$, $v$ will have an available color.
\end{proof}

\noindent
We can now complete our proof of the degree-choosability theorem.

\begin{proof}[Proof of~\Cref{thm:degree-choosability}]
    By~\cite{erdos1979choosability}, Gallai trees are not degree-choosable. Then it suffices to show that if $G$ is a connected graph and not a Gallai tree, $G$ is degree-choosable. Assume otherwise, and let $G$ be a counterexample with a minimum number of vertices. Let $L$ be a list assignment for $G$ with $|L(v)| \geq \deg(v)$ for all $v$, such that $G$ admits no $L$-coloring. Let $B$ be a block of $G$ which is neither a clique nor an odd cycle. By~\Cref{prop:rubin-subgraphs}, $B$ has no Rubin subgraph, as otherwise $G$ would have one as well. Then by~\Cref{lemma:chordal-graphs-degree-choosable}, $B$ is not chordal, so $B$ contains an odd induced cycle $C = v_1, v_2, \ldots, v_k, v_1$ of length $k \geq 5$.

    Since $B \neq C$, assume without loss of generality that $v_1$ has at least one neighbor in $B \setminus C$. Among all such neighbors of $v_1$, let $x$ be one with minimum distance to $C - v_1$ in $B - v_1$. Let $P = x, \ldots, v_i$ be a shortest path from $x$ to $C - v_1$ in $B - v_1$. Let $y$ be $v_i$'s neighbor on~$P$ (it is possible that $y = x$).
    Note that by our choice of $x$, $v_1$ has no neighbors among the internal vertices of $P$. Furthermore, since $P$ was chosen to be a shortest path, $y$ is the only vertex on $P$ with any neighbor on $C - v_1$. 
    
    It will be convenient to assume that $v_1$ and $v_i$ are not adjacent, i.e., $i \notin \{2, k\}$.
    To see that this assumption is safe, we begin by claiming that $y$ must have multiple neighbors on $C - v_1$.
    Otherwise, by combining the path $v_1, x, P, v_i$ with the two $v_iv_1$-paths given by the cycle $C$, we obtain a pair of induced cycles. Since $C$ is odd, one of these cycles must be even, and thus we have found a Rubin subgraph, a contradiction.
    Next, we claim that $y$ is non-adjacent to at least one of $v_2$ and $v_k$. Otherwise, if $y$ neighbors both $v_k$ and $v_2$, then the induced subgraph $G[\{v_k, v_1, v_2, y\}]$ is either a diamond (if $v_1$ and $y$ are adjacent, i.e., if $x = y$) or a $C_4$ (if $v_1$ and $y$ are not adjacent). Both are Rubin subgraphs, a contradiction.
    Thus, since $y$ has multiple neighbors along $C - v_1$ and is non-adjacent to at least one of $v_2$ and $v_k$, there must exist a neighbor of $y$ in $C - \{v_k, v_1, v_2\}$. Choosing this neighbor as $v_i$ satisfies all of the desired properties of $P$.

    We now define an auxiliary graph $H$, and an odd cycle $C_H$ in $H$. See~\Cref{fig:cycle-path-pair}.
    \begin{align*}
        H &= 
        \begin{cases}
            G \setminus \{v_{i+1}, \ldots, v_k\} + v_1v_i \quad\phantom{}\text{ if $i$ is odd.}\\
            G \setminus \{v_2, \ldots, v_{i - 1}\} + v_1v_i \quad\text{ if $i$ is even.}
        \end{cases} \\
        C_H &=
        \begin{cases}
            v_1, C, v_i, v_1 \quad\text{ if $i$ is odd.}\\
            v_i, C, v_1, v_i \quad\text{ if $i$ is even.}
        \end{cases}
    \end{align*}
    % cycle-figure.tex
% Requires in the main document's preamble:
%   \usepackage{tikz}
%   \usepackage{subcaption}
%   \usetikzlibrary{calc}

% ---------- Macro: one copy of the figure ----------
% #1 = +1 for "curve/gray right" version, -1 for "curve/gray left" version
% #2 = parity word only, e.g. "even" or "odd"
\newcommand{\cyclefigure}[2]{%
\begin{tikzpicture}[
  vtx/.style={circle, draw, fill=black, inner sep=0pt, minimum size=4pt},
  scale=0.6, every node/.append style={scale=0.50},
  every node/.style={font=\small}
]

\def\n{9}        % length of C (choose any odd n >= 5)
\def\R{2.1cm}
\pgfmathsetmacro{\side}{#1}
\ifnum#1=1 \def\iidx{5} \else \def\iidx{6} \fi

\foreach \k in {1,...,\n}{
  \pgfmathsetmacro{\ang}{90-(\k-1)*360/\n}
  \coordinate (V\k) at (\ang:\R);
}
\pgfmathtruncatemacro{\nextk}{\iidx+1}
\pgfmathtruncatemacro{\prevk}{\iidx-1}

\ifnum#1=1
  \draw[thick] (V1) \foreach \k in {2,...,\iidx}{ -- (V\k) };
  \draw[thick, gray] (V\iidx) \foreach \k in {\nextk,...,\n}{ -- (V\k) } -- (V1);
\else
  \draw[thick, gray] (V1) \foreach \k in {2,...,\iidx}{ -- (V\k) };
  \draw[thick] (V\iidx) \foreach \k in {\nextk,...,\n}{ -- (V\k) } -- (V1);
\fi

\ifnum#1=1
  \foreach \k in {2,...,\prevk}{ \node[vtx] at (V\k) {}; }
  \foreach \k in {\nextk,...,\n}{ \node[vtx, gray] at (V\k) {}; }
\else
  \foreach \k in {2,...,\prevk}{ \node[vtx, gray] at (V\k) {}; }
  \foreach \k in {\nextk,...,\n}{ \node[vtx] at (V\k) {}; }
\fi
\node[vtx] at (V1) {};
\node[vtx] at (V\iidx) {};

\node[left, yshift=4pt] at (V1) {$v_1$};
\ifnum#1=1
  \node[right] at (V2) {$v_2$};
  \node[left, gray] at (V\n) {$v_k$};
\else
  \node[right, gray] at (V2) {$v_2$};
  \node[left] at (V\n) {$v_k$};
\fi
\ifnum#1=1
  \node[below right, yshift=-4pt] at (V\iidx) {$v_i$};
\else
  \node[below left, yshift=-4pt] at (V\iidx) {$v_i$};
\fi
\node at (0,-3.1) {\footnotesize $(i$ \textrm{#2}$)$};

% --- dashed chord v1 -- v_i ---
\draw[thick, dashed] (V1) -- (V\iidx);

% --- external vertex x, adjacent to v1 ---
\coordinate (X) at (0,3.5);
\node[vtx] at (X) {};
\node[above] at (X) {$x$};
\draw[thick] (V1) -- (X);

% --- path P: now dash-dot, to distinguish it from the dashed v1--v_i chord ---
\coordinate (Bulge) at (\side*2.8,0.8);
\pgfmathsetmacro{\angI}{90-(\iidx-1)*360/\n}
\pgfmathsetmacro{\angY}{\angI+\side*16}
\coordinate (Y) at (\angY:2.9cm);

\draw[thick, dash dot] plot[smooth, tension=1] coordinates {(X) (Bulge) (Y)};
\draw[thick, dash dot] (Y) -- (V\iidx);
\node at ($(Bulge)+(\side*0.55,0.15)$) {$P$};

\node[vtx] at (Y) {};
\node[below] at (Y) {$y$};

% \node at ($(0,0)+(0.68*\side,0.75)$) {$C_H$};

\end{tikzpicture}%
}

% ---------- The actual figure ----------
\begin{figure}[t]
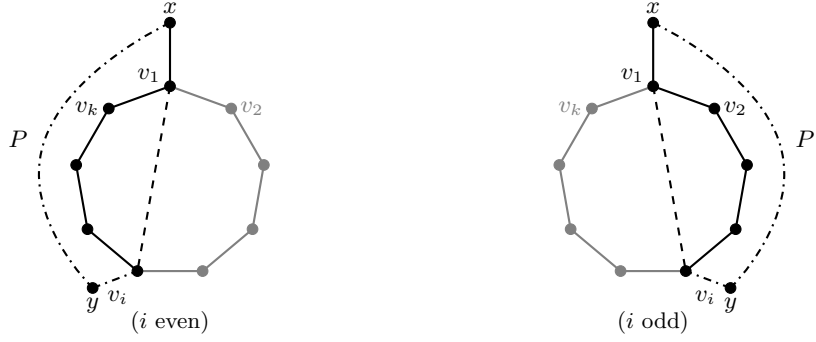

\centering
\begin{subfigure}{0.48\textwidth}
\centering
\cyclefigure{-1}{even}
\end{subfigure}
\hfill
\begin{subfigure}{0.48\textwidth}
\centering
\cyclefigure{1}{odd}
\end{subfigure}
\caption{Our proof of~\Cref{thm:degree-choosability} makes use of an induced odd cycle $C$ of length $k \geq 5$; we have illustrated the case where $k = 9$. The vertex $x$ is a neighbor of $v_1$, and $P$ (dot-dashed) is a shortest path from $x$ to $C - v_1$ not using $v_1$ (it is possible that $y = x$). The vertex $v_i$ is the endpoint of~$P$ on~$C$, and the two subfigures illustrate the cases where $i$ is even (left) or odd (right). The auxiliary graph $H$ considered in the proof is formed by deleting the grayed portion of the cycle and adding the edge $v_1v_i$. $H$ contains an induced odd cycle formed by combining the new edge $v_1v_i$ with the non-deleted portion of the original cycle $C$. The path $P$ is used in the proof to show that $H$ is not a Gallai tree.}
\label{fig:cycle-path-pair}
\end{figure}

    Let $H_1, \ldots, H_q$ be the connected components of $H$, and assume $H_1$ contains $C_H$. Let $B_1$ be the block of $H_1$ containing $C_H$.
    Note that $B_1$ contains both the cycles $C_H$ and $v_1, x, P, v_i, v_1$, so $B_1$ is not an induced odd cycle.
    Moreover, we claim that $B_1$ is not a clique. If $P$ has length greater than 1 then this is immediate, since by our choice of $x$, $v_1$ is not adjacent to any internal vertex of $P$. If $P$ has length $1$, then since $2 < i < k$, $B_1$ must not be a clique as otherwise either $\{v_1, v_2, v_3, x\}$ ($i$~odd) or $\{v_1, v_k, v_{k-1}, x\}$ ($i$~even) induce a diamond in $G$. 
    
    Thus, $H_1$ is not a Gallai tree. For each $j \in [q]$, Let $L_j$ be the restriction of $L$ to $V(H_j)$. Observe that for all $v \in V(H_1)$, $\deg_{H_1}(v) \leq \deg_G(v)$; the only interesting cases are $v_1$ and $v_i$, which have both become endpoints of the edge $v_1v_i$, but which also both had a neighbor deleted when we formed $H$. Consequently, $L_1$ is a degree-list assignment for~$G_1$. Furthermore, $H_1$ has strictly fewer vertices than $G$, so by the minimality of $G$, $H_1$ admits an $L_1$-coloring.
    We now consider $H_j$, for $j > 1$. $H_j$ contains at least one vertex $w$ which neighbors a vertex of $C \setminus H$ in $G$. Then $|L_j(w)| > \deg_{H_j}(w)$, so by~\Cref{lemma:degree-plus-one-choosable}, $H_j$ is $L_j$-choosable.
    Then $H$ is $L_H$-choosable, where $L_H$ is the restriction of $L$ to $H$.

    Fix an $L_H$-coloring $\phi_H$ of $H$. We obtain a final contradiction by finding an $L$-coloring $\phi_G$ of $G$. 
    For each $v \in V(G \setminus C)$, set $\phi_G(v) = \phi_H(v)$. For each $v_1, v_2, \ldots, v_k \in V(C)$, define $L^*(v_j) \vcentcolon = L(v_j) \setminus \{\phi_H(u) \mid uv_j \in E(G) \setminus E(C)\}$. 
    Note that it now suffices to find an $L^*$-coloring of $C$. Notice also that $|L^*(v_j)| \geq 2 = \deg_C(v_j)$ for each $j \in [k]$. 
    We may assume that the lists $L^*$ are all pairwise equal, as otherwise~\Cref{lemma:2-connected-unequal-degree-choosable} provides an $L^*$-coloring of $C$. Then we claim that $|L^*(v_j)| \geq 3$ for all $j \in [k]$, which will complete the proof by~\Cref{lemma:degree-plus-one-choosable}.
    
    Since $C_H$ is an odd cycle in $H$, $|\bigcup_{v \in V(C_H)} \phi_H(v)| \geq 3$. Also, because $\phi_H$ is proper, it follows from our definitions that $\phi_H(v) \in L^*(v)$ for each $v \in V(C_H)$. Then since $V(C_H) \subseteq V(C)$ and all lists are equal, we have $|L^*(v_j)| \geq 3$ for all $j \in [k]$, as desired.
\end{proof}

\section{Degenerate Partitioning}\label{sec:degenerate-partitioning}

In this section, we consider degenerate partitioning, a generalization of coloring. Given a vector $p = (p_1, \dots, p_s)$, for $s \geq 1$, of non-negative integers, we say a graph is \emph{$p$-partitionable} if there is a partition of $V(G)$ into sets $\{V_1,\dots,V_s\}$ so that $G[V_i]$ is $p_i$-degenerate for all $i \in [s]$. We call $\{V_1, \dots, V_s\}$ a \emph{$p$-partition} of $G$.
We consider the following problem.

\begin{quote}
\noindent \partition 
\vspace{.5em}\\
\textbf{Input:} A graph $G = (V,E)$ and integers $p=(p_1, \dots, p_s)$ with $p_i \geq 0$ for $i \in [s]$.
\\
\textbf{Task:} Decide if $G$ is $p$-partitionable.
\end{quote}

\noindent
If $p$ is fixed and not given as part of the input, we denote the problem to determine whether a graph $G$ is $p$-partitionable by \ppartition.
To see that \partition is indeed a generalization of coloring, consider the vector $p = 0^s$. Then \partition corresponds to asking whether $G$ can be partitioned into $s$ independent sets, i.e., whether $G$ is $s$-colorable.

In~\Cref{sec:quant-degenerate-partitioning}, we give a quantitative result showing that for certain vectors $p$ there are exponentially many $p$-partitions.
On the other hand, for many choices of $p$ this is not true.
Indeed, unlike the other settings considered in this paper, it is not even easy to characterize those graphs admitting a single $p$-partition:
Abu-Khzam, Feghali, and Heggernes~\cite{abu2020partitioning} have shown that for some vectors $p$, \ppartition is NP-hard.
In~\Cref{sec:hardness-degenerate}, we build on their result to show NP-hardness for a much broader parameter regime; see~\Cref{thm:hardness-degpart}.

\subsection{A Quantitative Degenerate Partitioning Theorem}\label{sec:quant-degenerate-partitioning}

We start with our quantitative result.
In \cite{corsini2023partitioning}, it has been shown that a greedy algorithm similar to the one for coloring leads to a valid partition for certain degeneracy values. As for our quantitative results for coloring, our partitioning procedures rely on the observation that in the greedy algorithm, for many vertices several colors are available. Note that unlike in our quantitative results for coloring, in this section we count partitions, see also \cref{rem:unique-cols}.

\begin{algorithm}
    \caption{Greedy degenerate
 partitioning}\label{alg:greedydegenerate}
    \SetKwInOut{input}{Input}\SetKwInOut{output}{Output}
    \SetKw{Break}{break}

    \input{A graph $G$ with maximum degree $\Delta$, a $(\Delta-1)$-degeneracy ordering of the vertices $v_1, \dots, v_n$, and integers $p_1, \dots, p_s$ with $\sum_{i=1}^s p_i=\Delta-s$.}
    \output{A partition of $V(G)$ into sets $(V_1, \dots, V_s)$ such that each $G[V_i]$ is $p_i$-degenerate.}

    \For{$i \leftarrow 1$ \KwTo $n$}{
        \For{$j \leftarrow 1$ \KwTo $s$}{
            \If{$v_i$ \text{has at most} $p_j$ {left neighbors in} $V_j$}{
                Add $v_i$ to $V_j$\;
                \Break\;
            }
        }
    }
    \Return ($V_1, \dots V_s$)\;
\end{algorithm}

\begin{lemma}[\cite{corsini2023partitioning}]\label{lemma:partition-algo}
    Let $G$ be a graph with maximum degree $\Delta$. Given a vector of $s \geq 1$ integers $p=(p_1, \dots, p_s)$ with $p_i \geq 0$ for $i \in [s]$ and $\sum_{i=1}^s p_i = \Delta - s$ and a $(\Delta-1)$-degeneracy
    ordering $v_1, \dots, v_n$ of $V(G)$, \cref{alg:greedydegenerate} partitions $V(G)$ into sets $\{V_1, \dots, V_s\}$ so that $G[V_i]$ is $p_i$-degenerate.
\end{lemma}

\noindent
Similar to Lemma~\ref{lem-kdeg2}, we will now show that in the greedy algorithm for degenerate partitioning, for many vertices there are several partition sets that they can safely join. To quantify this choice and obtain a bound on the number of partitions in a degenerate graph, we need the following definition.
For $p=(p_1, \dots, p_s)$, with $p_1 \geq \dots \geq p_s$, and a graph $G$ with maximum degree $\maxdeg$, define $\setsCanJoin{q}$, for some integer $q < \maxdeg$, as the smallest integer such that
$\sum_{i=1}^{\setsCanJoin{q}} (p_i+1) \geq \frac{\maxdeg-q}{2}$.

\begin{lemma}\label{lemma:irregular-degenerate} 
Let $G$ be a $d$-degenerate graph of maximum degree $\Delta$ with $d<\Delta$. Let $p=(p_1, \dots, p_s)$ be a vector of $s \geq 1$ integers with $p_i \geq 0$ for $i \in [s]$.
Let $q < \maxdeg$.
 If $\sum_{i=1}^{s} p_i = \Delta - s$, then $G$ admits at least $\setsCanJoin{q}^{\boundvlowshortb-s}$ $p$-partitions.
\end{lemma}
\begin{proof}
    Let $G$ be a $d$-degenerate graph for some $d < \Delta$. Fix a $d$-degeneracy ordering $\sigma$ of the vertices. We partition $V(G)$ by the left degree with respect to $\sigma$. Let \defvhighb, and let \defvlowb. Applying~\Cref{lem-highlow2} with parameter $k = \maxdeg$, we have \boundvlowb. 
    
    We get a lower bound on the number of $p$-partitions into sets $\{V_1, \dots, V_s\}$ by processing the vertices in the order of $\sigma$ from left to right. First, we assign $v_i$ to $V_i$ for $i \in [s]$. Note that this ensures the obtained partitions are unique. Now, for each following $v_j$ with $j \in \{(s+1), \dots, n\}$, 
    we observe that if $v_j$ has at most $p_i$ neighbors in $V_i$, then $v_j$ can be assigned to $V_i$. If $v_j \in \vlow$, then $\deg_G^\ell(v_j) \leq \frac{\maxdeg+q}{2}$. Hence, then the number of partite sets, $v_j$ can join is at least $\setsCanJoin{q}$, since $\sum_{i=1}^{\setsCanJoin{q}} (p_i +1) \geq \frac{\maxdeg-q}{2} = \maxdeg - \frac{\maxdeg+q}{2}$. To see this, note that if $v_j$ has $p_i+1$ neighbors in each of the $s-\setsCanJoin{q}$ smallest sets, the number of sets it can be assigned to is minimal.
    Therefore, there are at least $\setsCanJoin{q}^{|\vlow|-s} \geq \setsCanJoin{q}^{\boundvlowshortb-s}$ $p$-partitions.
\end{proof}

\begin{remark}
    If we set $q = 1$, then as in \cref{lem-kdeg2} we obtain a bound of $\setsCanJoin{q}^{\boundvlowshort-s}$. Note further that this bound is only exponential if $\setsCanJoin{q} \geq 2$, that is, only if $p_1 < \frac{\maxdeg-q}{2}$.
\end{remark}

\begin{figure}
     \begin{tikzpicture}
	\begin{scope}
		\node[vertex](c) at (0,0){};
	\node[rounded rectangle, minimum width=3.2cm, minimum height = 1.5cm, draw = black](g) at (1.5,0){$G_2$};

	\begin{scope}[shift = {(-1,0)}]
		\foreach \i in {1,..., 5}{
			\node[vertex](u\i) at (60*\i:1){};
		}
		\foreach \i in {1,...,4}{
			\foreach \j in {\i, ...,4}{
				\draw[edge](u\i)--(u\j);
			}
		}
		\foreach \j in {2,...,4}{
			\draw[edge](u5) -- (u\j);
		}

		\node[](t) at (-1,1){$G_1$};
	\end{scope}
	\draw[edge](u1) -- (c) -- (u5);
	
	\end{scope}

	\begin{scope}[shift = {(6,0)}]
		\node[rectangle, minimum width = 1cm, minimum height = 0.7cm, rounded corners = 5pt, draw=black](x) at (0.2,0.3){$X$};
		\node[vertex, label = right:$z$](z) at (0.2,-0.7){};
		\node[vertex](u) at (-0.5,-0.7){};
		\draw[edge](u) -- (z);
		%\draw[edge](z) -- (x.south west);
		\draw[edge](z) -- (-0.1,-0.05);
		\draw[edge](z) -- (0.5,-0.05);
		\draw[edge](z) -- (x.south);
		
	\node[rounded rectangle, minimum width=3.2cm, minimum height = 2cm, draw = black](g) at (0,0){};
	\node[](t) at (-1,0){$G_1$};
	\node[rounded rectangle, minimum width=3.2cm, minimum height = 1.5cm, draw = black](g) at (3,0){$G_2$};
		\node[vertex](c) at (1.5,0){};
	\end{scope}
\end{tikzpicture}
    \caption{\label{fig:degen-part} The two cases considered in the proof of Theorem~\ref{thm-degpart}: Case 1 (left), where $G_1$ is a quasi-clique and Case~2 (right), where $G_1$ contains a special neighborhood.}
\end{figure}
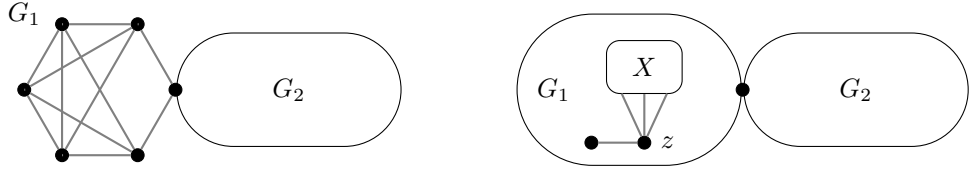

\noindent
It remains to consider the case where $G$ is $\maxdeg$-regular.
We introduce two further concepts which have been used in \cite{corsini2023partitioning} before. We say that a graph~$G$ is a \emph{quasi-clique}, denoted $\quasi$, if $G$ is the complete graph on $\Delta+1$ vertices with exactly one subdivided edge. In a graph $G$ with $z \in V(G)$ and $X \subseteq V(G)$, we say that the pair $(z, X)$ is a \emph{special neighborhood} if $X \subseteq N_G(z)$, $|X| = \Delta - 1$, $G[X]$ is not a complete graph, and $G - X$ is connected. In~\cite{corsini2023partitioning}, it has been shown that special neighborhoods exist in a useful class of graphs.

\begin{lemma}[\cite{corsini2023partitioning}] \label{lemma:special-neighborhood}
    Let $G$ be a 2-connected graph distinct from $\quasi$ and $K_{\Delta+1}$ and $v \in V(G)$. Suppose all vertices in $G$ have degree $\Delta$ except possibly $v$. Then there exists a special neighborhood $(z, X)$ with $v \notin X$.
\end{lemma}

\noindent
We now have all of the tools we need to prove the main result of this section.
\begin{theorem}\label{thm-degpart}
Let $G$ be a graph of maximum degree $\Delta$ which is not complete. Let $p=(p_1, \dots, p_s)$ be a vector of $s \geq 1$ integers with $p_1 \geq \dots \geq p_s \geq 0$, and let $q < \maxdeg$.
If $\sum_{i=1}^{s} p_i = \Delta - s$, then $G$ has at least $\setsCanJoin{q}^{\boundvlowshortb-p_s-s-3}$ $p$-partitions.
\end{theorem}
\begin{proof}
Suppose $G$ is $d$-degenerate for some $d < \maxdeg$. Then, we apply \cref{lemma:irregular-degenerate} and obtain $\setsCanJoin{q}^{\boundvlowshortb-s}$ $p$-partitions. Hence, we may assume that $G$ is $\maxdeg$-regular.
Let $G_1$ be a leaf block in the block-cut-tree of $G$. Note that $G_1$ = $G$ if $G$ is 2-connected. We distinguish two cases.
    
    \case{1} $G_1$ is a quasi-clique.\\
    Since $G_1$ is not $\maxdeg$-regular, $G_1 \neq G$. Let $u \in V(G_1)$ be the single cut vertex of $G_1$. 
    Let $G_2 = G - (V(G_1) \setminus \{u\})$.
    Since $G_1$ is a quasi-clique and $u$ has a neighbor in $G_2$, $\deg_{G_1}(u)=2$, and since $G$ is $\maxdeg$-regular, for all $v \in V(G_1) \setminus \{u\}$, we get $\deg_{G_1}(v)= \Delta$. 
    By \cref{obs-reg-vertremoval-deg}, $G_2$ is $(\Delta-1)$-degenerate, so by \cref{lemma:irregular-degenerate}, there are $\setsCanJoin{q}^{\boundvlowshortG{G_2}-s}$ $p$-partitions of $G_2$.
    Fix some $p$-partition $\partitionNotation{G_2}=\{V_1, \dots, V_s\}$ of $G_2$. We will get a lower bound on the number of partitions in $G_1$ compatible with $\partitionNotation{G_2}$.
    
    We construct an ordering of the vertices in $G_1$. Let $x,y$ be the neighbors of $u$ in $G_1$ and let $X = \{u, x,y\}$. We apply \cref{lemma:non-cutset-ordering} to $G_1$ and $X$ and obtain an ordering $\sigma$ where the vertices in $X$ are the leftmost vertices. Let $v$ be the rightmost vertex according to $\sigma$. We order the vertices in $X$ $u,x,y$. Note that this does not influence the properties of the ordering obtained by \cref{lemma:non-cutset-ordering}.
    We apply \cref{lem-highlow2} to $G_1$ with ordering $\sigma$ and obtain that \boundvlowG{G_1}, for \defvlowG{G_1}.
    
    To count the number of partitions of $G_1$, we use an argument similar to \cref{lemma:irregular-degenerate}. We process the vertices in the order given by $\sigma$. Since we want to extend a partition of $G_2$ and $u \in V(G_1) \cap V(G_2)$, we assign $u$ to the same set as for $G_2$. Next, we assign $x$ to a different set than the one containing $u$; any set will do, since $u$ is the only neighbor of $x$ which has already been assigned. Further, we assign $y$ to the same set as $x$, which is once again possible because $u$ is the only neighbor of $y$ which has already been assigned.
    
    For each remaining vertex $w \neq v$, there is guaranteed to be at least one available set $V_i$ with at most $p_i$ left neighbors of $w$, since each such $w$ has left degree less than $\maxdeg$ by Property (3) of~\Cref{lem-highlow2}.

    It remains to show that we can safely add the righmost vertex $v$ to some set. Let $V_i$ be the set containing both $x$ and $y$. If $p_i = 0$, then there must be an available set for $v$ because $V_i$ contains $p_i + 2$ neighbors of $v$.
    Then we may assume that $p_i > 0$. In this case, we add $v$ to $V_i$. To see that $G[V_i]$ remains $p_i$-degenerate, observe that $|V_i \setminus \{v\}| \leq p_i + 1$, as otherwise we have contradicted our greedy coloring procedure (since $G_1$ is a quasi-clique). Then $|V_i| \leq p_i + 2$. A $p_i$-degeneracy ordering may be obtained by taking an arbitrary order with $x$ rightmost. Since $|V_i| \leq p_i + 2$, it is immediate that every vertex except $x$ has at most $p_i$ left neighbors. Meanwhile, $x$ has at most $p_i$ left neighbors because it is non-adjacent to $y$.

    Thus, for each partition of $G_2$, there are at least $\setsCanJoin{q}^{\boundvlowshortG{G_1}-3}$ partitions of $G_1$ that extend that of $G_2$.
    It follows that the total the number of $p$-partitions of $G$ is at least 
    \[\left( \setsCanJoin{q}^{\boundvlowshortG{G_2}-s} \right) \left( \setsCanJoin{q}^{\boundvlowshortG{G_1}-3} \right) = \setsCanJoin{q}^{\boundvlowshortG{G}-s-3}.\]

    \case{2} $G_1$ is not a quasi-clique.\\
    If $G \neq G_1$, then let $v \in G_1$ be the single cut vertex of $G$. Otherwise, let $v$ be any vertex of~$G_1$.
    By \cref{lemma:special-neighborhood}, since $G_1$ is neither complete nor a quasi-clique there exists $z \in V(G_1)$ and $X \subseteq G_1$ such that $(z, X)$ is a special neighborhood in $G_1$ with $v \notin X$. Since $v \notin X$, and $G_1-X$ is connected, $G-X$ is connected. Let $x, y \in X$ be two non-adjacent vertices that exist by definition of $X$. We apply \cref{lemma:non-cutset-ordering} to $G$, $X$ and $z$ and find an ordering~$\sigma$ with $z$ as the rightmost vertex. We let $x,y$ be the leftmost vertices of $X$ in $\sigma$.
     We partition the vertices of $G$ based on their left degree in $\sigma$. Let \defvhighG{G}, and let \defvlowG{G}. By Lemma~\ref{lem-highlow2}, \boundvlowG{G}.

    We get a lower bound on the number of $p$-partitions of $G$ by processing the vertices in the order of $\sigma = v_1, \dots, v_n$ from left to right.
    We assign $x,y$ as well as $v_3, \dots, v_{p_s+2}$ to $V_s$. Note that $V_s$ is $p_s$-degenerate, by the same argument as given in Case 1, since $x$ and $y$ are non-adjacent. Further, for $i \in \{1,\dots, s-1\}$, we assign $v_{p_s + i + 2}$ to $V_i$. For all following vertices, similar to Case~1, there are at least $\setsCanJoin{q}$ partite sets a vertex in $\vlow$ can be assigned to, and at least one set a vertex in $\vhigh$ can be assigned to.
    Finally, since $z$ has $p_s+2$ neighbors in $V_s$ and $\sum_{i=1}^s(p_i+1) = \maxdeg$, there is some set $V_i$ containing at most $p_i$ neighbors of $z$.
    So we can safely assign $z$ to $V_i$.

    Hence, the number of $p$-partitions of $G$ is at least 
    \[\setsCanJoin{q}^{|\vlow| - (p_s + s+1)} \geq \setsCanJoin{q}^{\boundvlowshortG{G}- p_s-s -1} .\]

    \noindent
    The result now follows by a pessimistic combination of the bounds obtained in the two cases.
\end{proof}

\begin{remark}
    If $p_i = d$ for $d \leq \frac{\maxdeg}{2}$, for all $i \in [s]$, then for $q=1$, we obtain that $\setsCanJoin{1} = \lceil \frac{s}{2}\rceil$. Hence, the number of partitions is at least $\lceil\frac{s}{2}\rceil^{\frac{|V(G)|}{\maxdeg+1} - d-s -3}$. Setting $d = 0$, that is, requiring the partition to be a coloring, we recover (up to the additive constant) our bound from \cref{thm:quantbrooks} in the modified version for unique colorings, see \cref{rem:unique-cols}.
\end{remark}

\subsection{Hardness for Degenerate Partitioning}\label{sec:hardness-degenerate}

We prove \cclass{NP}-completeness for \ppartition
for graphs of maximum degree~$\Delta$ bounded below by a function of $p$ and $s$. 
Our goal will be to improve upon the existing hardness result given by Abu-Khzam, Feghali, and Heggernes~\cite{abu2020partitioning}.

\begin{theorem}[\cite{abu2020partitioning}]
    \textsc{\ppartition}, for $p = (p_1, p_2) \neq (1, 1)$, is \cclass{NP}-complete for graphs of maximum degree $\maxdeg = p_1 + p_2 + 3 \geq 5$.
\end{theorem}

To achieve this, we
reduce from \kcolor, which is known to be \cclass{NP}-complete for graphs of bounded degree by a result from Molloy and Reed \cite{molloy2014almost}. We derive a degree bound from this result that we state in \cref{cor-coloringhardbound}. A short proof can be found in the appendix.

\begin{theorem}[\cite{molloy2014almost}]\label{thm:molloy-hardness}
    Let $\Delta \geq 3$ and let $k_\Delta$ be the maximum integer such that $\Delta \geq (k_\Delta+1)(k_\Delta+2)$. For every $3 \leq k \leq \Delta-k_\Delta-1$, \textsl{\kcolor} is \cclass{NP}-complete for graphs of maximum degree $\Delta$.
\end{theorem}

\begin{restatable}{corollary}{corcoloringhard}\label{cor-coloringhardbound}
    For every integer $k \geq 3$ and $\Delta \geq \left(\frac{1}{2} + \frac{\sqrt{4k+1}}{2}\right)^2$, \textsl{\kcolor} is \cclass{NP}-complete for graphs with maximum degree $\Delta$.
\end{restatable}

\noindent
 We denote $\maxdeg_k = \left\lceil \left(\frac{1}{2} + \frac{\sqrt{4k+1}}{2}\right)^2  \right\rceil$, so \kcolor is \cclass{NP}-complete for graphs of maximum degree~$\maxdeg_k$.
We are now set to prove our result.
The key element of the construction are the vertex gadgets which force adjacent gadgets to be put in different partite sets, thus leading to different colors.

\begin{theorem}\label{thm:hardness-degpart}
    \textsl{\ppartition}, for $p = (p_1, \dots, p_s)$, is \cclass{NP}-complete for graphs of maximum degree $\Delta \geq \sum_{i=1}^sp_i+ s + \maxdeg_s$.
\end{theorem}
\begin{proof}
    Let $G=(V, E)$ be an instance of \scolor. 
    We construct an instance of \ppartition with a graph $H$. 
    We initialize $H$ as the empty graph. 
    For every vertex $v$ in $G$, we add a \emph{vertex gadget}
    $S_v$ to $H$ (see Figure \ref{fig:hardness-degpart}).
    The vertex gadget $S_v$ consists of $\sum_{i=1}^s p_i+s-1$ \emph{clique} vertices which form a clique and two \emph{external} vertices $x_v$ and $y_v$ which are adjacent to every clique vertex. 
    For each edge $uv$ in $G$, we add $x_ux_v$ and $y_uy_v$ to $E(H)$. 
    We call $S_v$ and $S_u$ \emph{matched} gadgets. 
    This completes the construction. 

    Now, we compute the maximum degree of $H$. Let $v \in V(S_u)$ for a vertex gadget $S_u \subseteq H$. If $v$ is a clique vertex, then $\deg_H(v) = \sum_{i=1}^s p_i + s$. If $v$ is an external vertex, then $v$ has exactly $\sum_{i=1}^sp_i+s-1$ neighbors which are clique vertices in $S_u$ and at most $\Delta_s$ neighbors which are external vertices in vertex gadgets distinct from $S_u$. Hence, $\deg_H(v) \leq \sum_{i=1}^s p_i + s + \Delta_s = \Delta(H)$.

    We claim that $H$ is $p$-partitionable if and only if $G$ is $s$-colorable. 
    First, assume $H$ is $p$-partitionable. We will show $G$ is $s$-colorable. Let $\mathcal{P}$ be a $p$-partition of $V(H)$ into sets $V_1, \dots, V_s$. We begin by proving some important claims about the vertex gadget.

    \begin{figure}
    \centering
    \begin{tikzpicture}
\newcommand{\vertexgadget}[4]{
	\node[vertex, label = above:#3](#1) at (0,1){};
	\node[rectangle, minimum width=3cm, minimum height = 0.7cm, rounded corners = 5pt, draw = black](C) at (0,0){};
	\node[vertex, label = below: #4](#2) at (0,-1){};

	\draw[edge](#1) -- (1,0.35);
	\draw[edge](#1) -- (-1,0.35);
	\draw[edge](#1) -- (1/3,0.35);
	\draw[edge](#1) -- (-1/3,0.35);
	\draw[edge](#2) -- (1,-0.35);
	\draw[edge](#2) -- (-1,-0.35);
	\draw[edge](#2) -- (1/3,-0.35);
	\draw[edge](#2) -- (-1/3,-0.35);
}

\vertexgadget{xu}{yu}{$x_u$}{$y_u$}
\node[](t) at (-1.5,1){$S_u$};

\begin{scope}[shift = {(4,0)}]
\vertexgadget{xv}{yv}{$x_v$}{$y_v$}
\node[](t) at (1.5,1){$S_v$};
\end{scope}

\draw[edge](xu) -- (xv);
\draw[edge](yu) -- (yv);
\end{tikzpicture}
    \caption{\label{fig:hardness-degpart}Two vertex gadgets connected by a matching, used in the proof of Theorem~\ref{thm:hardness-degpart}. Each box represents a clique, the vertices $x_u, y_u, x_v, y_v$ are the external vertices.}
    \end{figure}
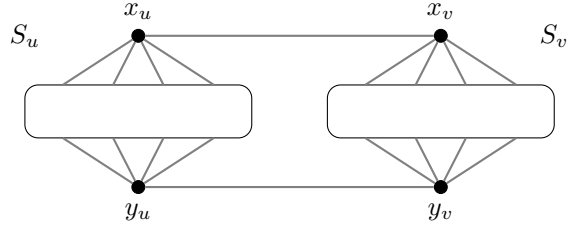

\begin{claim} \label{clm: external_same}
    Let $S_v$ be a vertex gadget in $H$. The external vertices of $S_v$ belong to the same set in $\mathcal{P}$.
\end{claim}
\begin{proof}
    Let $x$ and $y$ be the external vertices of $S_v$ with $x \in V_i$ and $y \in V_j$ for $i, j \in [s]$. Suppose toward a contradiction that $i \neq j$.
    If $V_i$ contains $p_i+1$ clique vertices of $S_v$, then the subgraph induced on $x$ and all clique vertices of $S_v$ in $V_i$ have minimum degree at least~$p_i+1$. Hence, at most $p_i$ clique vertices of $S_v$ are in $V_i$. 
    Likewise, at most $p_j$ clique vertices in $S_v$ are in $V_j$.
    Similarly, if there were $p_h+2$ clique vertices in $V_h$ for $h \in [s]$, $h\neq i, j$, then the subgraph induced on the clique vertices of $S_v$ in $V_k$ has minimum degree at least~$p_h+1$. Hence, at most $p_h+1$ clique vertices of $S_v$ are in $V_h$. Since
     \[\sum_{h=1, h \neq i,j}^s (p_h+1) + p_i + p_j = \sum_{i=1}^s p_i+s-2 < \sum_{i=1}^s p_i +s -1,\]
     there is a clique vertex of $S_v$ which is not contained in a set of $\mathcal{P}$, contradicting the fact that $\mathcal{P}$ is a partition.
    \end{proof}

\begin{claim} \label{clm: two_edges}
    Let $S_v$ and $S_u$ be matched gadgets in $H$. Let $x$ be an external vertex of $S_v$ with $x \in V_i$ and $y$ be an external vertex of $S_u$ with $y \in V_j$. Then, $j \neq i$. 
\end{claim}
\begin{proof}
    Suppose toward a contradiction that $i=j$. Note in both $S_v$ and $S_u$, there are exactly $p_i$ clique vertices in $V_i$. Thus, $G[V_i]$ has minimum degree $p_i+1$, so $G[V_i]$ is not $p_i$-degenerate and $\mathcal{P}$ is not a $p$-partition.
\end{proof}
    We color the vertices of $G$ with colors $i \in [s]$. For all $v \in V(G)$, there is a corresponding vertex gadget $S_v$ in $H$. 
    By Claim \ref{clm: external_same}, both external vertices of $S_v$ are in the same set $V_i$ for some $i \in [s]$. We assign $v$ the color $i$. We argue that this is a proper coloring. 
    Consider some edge $uv \in E(G)$. By construction, $S_v$ and $S_u$ are matched gadgets. Let $a \in V(S_v)$ be an external vertex with $a \in V_i$ and let $b \in V(S_u)$ be an external vertex with $b \in V_j$. By \cref{clm: two_edges}, $i \neq j$, so $u$ and $v$ are assigned different colors. Therefore, this is a proper coloring and thus, $G$ is $s$-colorable.
    
    \medskip
    \noindent
    We now suppose $G$ is $s$-colorable and show that $H$ is $p$-partitionable by assigning each vertex in~$H$ to some set $\mathcal{P}=\{V_1 \dots V_s\}$.
    Fix a proper coloring of $G$ with colors $i \in [s]$. For $v \in V(G)$ with color $i$, assign all external vertices of $S_v$ to the set $V_i$.
    Assign exactly $p_i$ clique vertices of $S_v$ to $V_i$ and $p_j+1$ clique vertices to $V_j$ for all $j \in [s]$ with $j \neq i$.

    To see that $\mathcal{P}$ is a $p$-partition, note first that $S_v[V_j]$ is a clique on $p_j+1$ vertices for all $j \in [s]$ with $j \neq i$, which is $p_j$-degenerate, and $S_v[V_i]$ is $p_i$-degenerate. 
    Now, consider two matched gadgets $S_u$ and $S_v$. 
    By construction, $uv \in E(G)$, so $u$ is assigned a different color from $v$. 
    Therefore, the external vertices of $S_u$ are in a set $V_i$ and the external vertices of~$S_v$ are in a set $V_j$ with $i \neq j$. 
    Hence, $S_v[V_i]$ is disconnected from $S_u[V_i]$ for all $i \in [s]$, so $(V_1, \dots, V_s)$ is a $p$-partition of $H$.
\end{proof}

\noindent
We can improve the maximum degree bound by optimizing the number of external vertices in the vertex gadget. 
We modify the construction of $H$ as follows. We allow the vertex gadget to have $\ell$ external vertices instead of 2. 
For every vertex gadget $S_v$, we distribute the edges matching $S_v$ to other vertex gadgets as evenly as possible among the external vertices to minimize the maximum degree in each gadget.

The degree of the clique vertices in $H$ is then $\sum_{i=1}^s p_i+s+\ell-2$. The degree of the external vertices in $H$ is at most $\sum_{i=1}^sp_i+s-1 + \left\lceil\frac{2}{\ell}\Delta_s\right\rceil$ since there are two edges in~$H$ matching vertex gadgets for every edge in $G$ distributed over $\ell$ external vertices. To minimize $\Delta(H)$, we therefore need to minimize $f(\ell)=\max \left \{\ell-1, \left\lceil\frac{2}{\ell}\Delta_k\right\rceil\right \}$. First, note the minimum of $g(\ell) = \max\{\ell-1, \frac{2}{\ell}\Delta_k\}$ with $\ell \in \mathbb{R}$ is achieved for
    $\ell^*=\frac{1+\sqrt{8k+4\sqrt{4k+1}+5}}{2}$.
    Since we require that $\ell \in \mathbb{N}$, we obtain that 
    the minimum value of $f(\ell)$ is $f_{min}=\min \{f(\lfloor \ell^* \rfloor), f(\lceil \ell^*\rceil)\}$. It then follows that \ppartition is \cclass{NP}-complete for graphs with maximum degree 
$\sum_{i=1}^s p_i+s-1+ f_{min}$.

%%
%% Bibliography
%%

%% Please use bibtex, 

\bibliography{refs.bib}

\appendix

\section{Proof of~\Cref{cor-coloringhardbound}}

\corcoloringhard*

\begin{proof}
    First, note that by the definition in the statement of~\Cref{thm:molloy-hardness}, we have $\maxdeg \geq (k_\maxdeg + 1)(k_\maxdeg + 2)$. It follows that
    \[
        \sqrt{\maxdeg} \geq \sqrt{(k_\maxdeg + 1)(k_\maxdeg + 2)} > k_\maxdeg + 1.
    \]
    That is, we have $k_\maxdeg < \sqrt{\maxdeg} - 1$. Now, let $k$ be any value such that $3 \leq k \leq \maxdeg - (\sqrt{\Delta} - 1) - 1$. By substituting $k_\maxdeg < \sqrt{\maxdeg} - 1$, we obtain that for all such $k$, $k \leq \maxdeg - k_\maxdeg - 1$. After rearranging and applying~\Cref{thm:molloy-hardness}, we have that \kcolor is \cclass{NP}-hard whenever $k \geq 3$ and $\maxdeg \geq k + (\sqrt{\maxdeg} - 1) + 1 = k + \sqrt{\maxdeg}$.
    For a fixed $k \geq 3$, to find the minimum~$\maxdeg$ for which this inequality holds we need to solve the equation $\maxdeg_0 = k + \sqrt{\maxdeg_0}$. The solution is
    \[ \maxdeg_0 = \left(\frac{1}{2} + \frac{\sqrt{4k+1}}{2}\right)^2.\]
\end{proof}

\section{Proof of~\Cref{prop:rubin-subgraphs-two-colorings}}\label{appendix:degree-choice}

\rubinsubgraphstwocolorings*
\begin{proof}%[Proof of~\Cref{prop:rubin-subgraphs-two-colorings}]
    Let $H$ be a Rubin graph on vertices $v_1, v_2, \ldots, v_n$, where we have arranged the vertices as they appear along the cycle (possibly with one chord) defining~$H$. Let $L$ be any degree-list assignment.
    We consider cases. In what follows, we say $v_{i - 1} = v_n$ if $i = 1$ and $v_{i+1} = v_1$ if $i = n$.

    \case{1} There exists a vertex $v_i$ with $|L(v_i)| > \deg(v_i)$.\\
    Without loss of generality, say $i = n$. Then we greedily color the vertices in order of increasing index. This is always possible, since $v_n$ has more available colors than neighbors. Moreover, there are at least two choices of color for $v_1$, yielding at least $2$ colorings.

    \case{2} All lists are pairwise identical.\\
    If $H$ has a chord then we are in Case 1, so we may assume that $H$ is an induced even cycle. Then we observe that there are two ways to color an even cycle by two colors.

    \case{3} There exists a vertex $v_i$ and a color $x$ such that $x \in L(v_i)$ and $x \notin L(v_{i - 1}) \cup L(v_{i+1})$.\\
    Without loss of generality, say $i = 1$. Fix color $x$ for $v_1$, and then color greedily in order of increasing index. Each vertex besides $v_n$ has at least one uncolored neighbor at the time that it is colored, and thus an available color. The vertex $v_n$ will have an available color since one of its neighbors received color $x \notin L(v_n)$. Finally, there are at least two choices of color for $v_2$ since $x \notin L(v_2)$.
    Notice that this argument works regardless of the endpoints of a chord, if one exists.

    \case{4} There exists a vertex $v_i$ and two distinct colors $x, y$ such that $x, y \in L(v_i)$, $x \notin L(v_{i+1})$, and $y \notin L(v_{i-1})$.\\
    Without loss of generality, say $i = 1$. Produce one coloring by fixing color $y$ for $v_1$ and then greedily coloring in order of increasing index. Note that $v_n$ has an available color, since one of its neighbors received color $y \notin L(v_n)$. 
    Produce a second coloring by fixing color $x$ for $v_1$ and greedily coloring in order of decreasing index. Now, $v_2$ has an available color since one of its neighbors received color $x \notin L(v_2)$.
    Notice again that this argument works regardless of the endpoints of a chord, if one exists.

    \case{5} There exists a vertex $v_i$ and two distinct colors $x, y$ such that $x, y \in L(v_i)$ and $x, y \notin L(v_{i - 1})$ or $x, y \notin L(v_{i+1})$.\\
    Without loss of generality, say $i = 1$, and suppose $x, y \notin L(v_n)$. Produce one coloring by fixing color $x$ for $v_1$ and then greedily coloring in order of increasing index. The vertex $v_n$ has an available color, since one of its neighbors received color $x \notin L(v_n)$. Produce a second coloring by fixing $y$ for $v_1$ and then coloring in order of increasing index. Once again $v_n$ has an available color, since one of its neighbors received color $y \notin L(v_n)$. The proof for the case where $x, y \notin L(v_{2})$ is identical, except that we greedily color in order of decreasing index. Finally, once again the argument does not depend upon the endpoints of a chord, if one exists.

    \medskip
    \noindent
    The cases above are not exhaustive, but they give us the tools we need to complete the proof. In what follows, we assume that none of these cases applies. 
    
    Assume first that $H$ has a chord, and say without loss of generality that $v_1$ is an endpoint of this chord. 
    Then $\deg(v_1) = 3$ and $\deg(v_2) = \deg(v_n) = 2$. Thus, if case 1 does not apply, then either case 3 or case 4 must apply.

    Then we may assume that $H$ is an induced even cycle with all lists of length $2$. Since Case 2 does not apply, not all lists are pairwise identical, and in particular we may assume without loss of generality that $L(v_1) \neq L(v_n)$. Because Case 5 does not apply, $L(v_1)$ and $L(v_n)$ have exactly one common color. Concretely, we may write $L(v_1) = \{a, b\}$ and $L(v_n) = \{b, c\}$. Because neither Case 3 nor 4 applies, we may assume that $L(v_2) = L(v_1) = \{a, b\}$.%, and similarly that $L(v_{n-1}) = L(v_n) = \{b, c\}$.
    
    We now consider two greedy procedures producing (not necessarily distinct) $L$-colorings $\phi_1$ and $\phi_2$ of $H$. We set $\phi_1(v_1) = a$. We then complete $\phi_1$ by greedily coloring in order of increasing index. There will be an available color for $v_n$, since $a \notin L(v_n)$. Notice that $\phi_1(v_2) = b$. Similarly, we form $\phi_2$ by setting $\phi_2(v_n) = c$ and then greedily coloring in order of decreasing index.

    If $\phi_1$ and $\phi_2$ differ on any vertex then we found two colorings, so we may assume that
    \begin{align}
        \phi_1(v_i) = \phi_2(v_i) \text{ for all } i \in [n].\label{eq:rubin-coloring:identical-colorings}
    \end{align}

    Now suppose there exists some vertex $v_i$, for $i > 1$, such that $\phi_1(v_{i - 1}) \notin L(v_i)$. Then $v_i$ had two available colors during the greedy coloring procedure defining $\phi_1$, and we may form a second coloring $\phi_1'$ by changing the color of $v_i$ and then continuing the greedy procedure as before.
    Thus, we may assume that
    \begin{align}
        L(v_i) = \{\phi_1(v_{i - 1}), \phi_1(v_i)\} \text{ for all } i > 1.\label{eq:rubin-coloring:list1}
    \end{align}
    By a symmetric argument, we may assume that
    \begin{align}
        L(v_i) = \{\phi_2(v_{i}), \phi_2(v_{i+1})\} \text{ for all } i < n.\label{eq:rubin-coloring:list2}
    \end{align}

    We now complete the proof by showing that equations (\ref{eq:rubin-coloring:identical-colorings}), (\ref{eq:rubin-coloring:list1}), and (\ref{eq:rubin-coloring:list2}) cannot all be true. Assume otherwise. Then by combining equations (\ref{eq:rubin-coloring:identical-colorings}), (\ref{eq:rubin-coloring:list1}), and (\ref{eq:rubin-coloring:list2}) we obtain the following.
    \begin{align*}
        L(v_i) = \{\phi_1(v_{i - 1}), \phi_1(v_i), \phi_1(v_{i+1})\} \text{ for all } 1 < i < n.
    \end{align*}
    Recalling that we have already assumed Case 1 does not apply, this implies that $\phi_1(v_{i-1}) = \phi_1(v_{i+1})$ for all $1 < i < n$.
    Since $\phi_1(v_1) = a$ and $\phi_1(v_2) = b$, it follows by induction that for all $i \in [n],$
    \begin{align*}
        \phi_1(v_i) = \begin{cases}
                        a \quad\text{if $i$ is odd}\\
                        b \quad\text{if $i$ is even}.\\
        \end{cases}
    \end{align*}

    Since $n$ is even, this implies that $\phi_1(v_n) = b$, and by equation~(\ref{eq:rubin-coloring:identical-colorings}) we have $\phi_2(v_n) = b$, but this contradicts the definition of $\phi_2$. This completes the proof.
\end{proof}

\end{document}